\documentclass[11pt]{article}

\usepackage{amsmath,amsthm}

\usepackage{amssymb,latexsym}

\usepackage{enumerate}

\newcommand{\FF}{{\mathcal F}}

\newcommand{\II}{{\mathcal I}}
\newcommand{\LL}{{\mathcal L}}

\newcommand{\NN}{{\mathcal N}}

\newcommand{\BD}{{\mathbb D}}

\newcommand{\BN}{{\mathbb N}}
\newcommand{\N}{{\mathbb N}}
\newcommand{\BR}{{\mathbb R}}
\newcommand{\Rp}{{\mathbb R_+}}

\newcommand{\fch}{{\mathbf{1}}}

\newtheorem{theorem}{\bf Theorem}[section]
\newtheorem{proposition}[theorem]{\bf Proposition}
\newtheorem{lemma}[theorem]{\bf Lemma}
\newtheorem{corollary}[theorem]{\bf Corollary}

\theoremstyle{definition}
\newtheorem{definition}[theorem]{Definition}
\newtheorem{example}[theorem]{\bf Example}

\newtheorem{remark}[theorem]{Remark}

\numberwithin{equation}{section}

\begin{document}

\title {Reflected Skorokhod equations and the Neumann boundary value problem for elliptic equations with L\'evy-type operators
}
\author {Andrzej Rozkosz and Leszek S\l omi\'nski
\\ {\small Faculty of Mathematics and Computer Science,
Nicolaus Copernicus University} \\
{\small  Chopina 12/18, 87--100 Toru\'n, Poland}}
\date{}
\maketitle
\begin{abstract}
We consider Neumann problem for linear elliptic equations involving integro-differential operators of L\'evy-type. We show that suitably defined viscosity solutions have probabilistic representations given in terms of the reflected stochastic Skorokhod equation associated with an It\^o process and an independent pure-jump L\'evy process.
As an application of the representation  we show that viscosity solutions arise as limits of some penalized equations and give some stability results for the viscosity solutions. Our proofs are based on new limit theorems for solutions of penalized stochastic equations with jumps and new estimates on the bounded variation parts of the solutions.
\end{abstract}
{\small \noindent{\bf Keywords:} Elliptic equation, Neumann problem, L\'evy-type operator, viscosity solution, sto\-chastic representation, reflected Skorokhod equation, penalization.
\smallskip\\
{\bf AMS MSC 2010:} 60H30, 35J25}

\section{Introduction}

Let $D$ be a bounded,  convex $C^2$-domain in $\BR^d$,
$f:\BR^d\rightarrow\BR$, $g:\BR^d\rightarrow\BR$ be measurable functions
and  $\lambda>0$. In the present
paper, we consider the Neumann problem which formally can be stated as follows:
find $u$ such that
\begin{equation}
\label{eq1.1}
-Lu+\lambda u=f\quad\mbox{in }D,\qquad
\frac{\partial u}{\partial\bar{\mathbf n}}=-g\quad\mbox{in }D^c,
\end{equation}
where $\bar{\mathbf n}(x)={\mathbf n}(\Pi(x))$,  ${\mathbf n}$
is the inward unit normal vector and $\Pi(x)$ is the projection of $x$ on $\bar D$. In (\ref{eq1.1}),  $L$ is an integro-differential operator of the form
\[
L=\LL+\II.
\]
We assume that its local part is of the form
\begin{equation}
\label{eq1.8}
\LL=\frac12\sum^d_{i,j=1}a_{ij}(x)
\frac{\partial^2}{\partial x_i\partial x_j}
+\sum^d_{i=1}b_i(x)\frac{\partial}{\partial x_i},\qquad a=\sigma\cdot\sigma ^T,
\end{equation}
for some  bounded Lipschitz-continuous $\sigma:\BR^d\rightarrow\BR^d\times\BR^d$, $b:\BR^d\rightarrow\BR^d$, and  the nonlocal part is given by
\[
\II[u](x)=\int_{\BR^d}(u(x+y)-u(x)-y\cdot\nabla u(x)\fch_{B(0,1)}(y))\,\nu(dy),\quad u\in C^2_0(\BR^d),
\]
for some L\'evy measure  $\nu$ on $\BR^d$.
In  \cite{PZ} it is proved that if $\nu=0$, i.e. the operator $L:=\LL$ is local, then for bounded continuous $f,g$ there exists a unique viscosity solution $u$ of (\ref{eq1.1}) (with boundary condition imposed on $\partial D$) and $u$ has the stochastic representation
\begin{equation}
\label{eq1.2}
u(x)=E\int^{\infty}_0e^{-\lambda t}(f(X^x_t)\,dt+g(X^x_t)\,d|K^x|_t),
\quad x\in\bar D,
\end{equation}
where $(X^x,K^x)$ is the unique solution of the reflected Skorokhod SDE in $\bar D$ associated with $\LL$ and $|K^x|_t$ is the total variation of $K^x$ on $[0,t]$. In fact, in \cite{PZ} semilinear equations are considered.

In case $\nu\neq0$  it is by no means clear what one should mean by a solution of (\ref{eq1.2}). This problem was addressed from the analytic point of view in
\cite{BGJ} (see also \cite{BCGJ} for the case of  half space). Quite different approach is adopted in \cite{DR-OV}. From the probabilistic perspective a natural idea is to define a solution of (\ref{eq1.1}) as the function defined by (\ref{eq1.1}) but with $(X^x,K^x)$ replaced by a solution of some reflected process associated with $L$. However, even in the case of fractional Laplacian, i.e., when $L=\II=-(-\Delta)^{\alpha/2}$ for some $\alpha\in(0,2)$, there are several
reasonable definitions of the reflected process.
The process $N$ associated with $-(-\Delta)^{\alpha/2}$ is the symmetric $\alpha$-stable process in $\BR^d$. For instance, as candidates for the title of  ``reflected $\alpha$-stable process'' one can mention the $\alpha/2$-subordinator of the reflecting Wiener process in $\bar D$ (see \cite{J3}), the process associated with  the actively reflected Dirichlet space corresponding to $-(-\Delta)^{\alpha/2}$ (see \cite{CF,BBC}) or
the process obtained by solving  the Skorokhod equation for $N$ in $D$ (see Section \ref{sec2}).  For reasons explained briefly below, in the present paper we provide some justification for the last choice.

In \cite{BGJ} a new definition of viscosity solution to Neumann problem for nonlocal equation including (\ref{eq1.1}) as a special case is given. Under natural assumptions it guarantees the existence and uniqueness of a viscosity solution. Furthermore, it has the property that the unique viscosity solution $u$ of (\ref{eq1.1}), in the sense of \cite{BGJ}, is the pointwise limit of solutions $u_n$ of some penalized problems in $\BR^d$. This convergence result provides additional justification for the definition of the Neumann-type boundary conditions adopted  in \cite{BGJ}.

In the present paper we show that the (unique) viscosity solution $u$ of (\ref{eq1.1}), in the sense defined in \cite{BGJ}, has the probabilistic representation
\begin{align}
\label{eq1.3}
u(x)&=E\int_{[0,\infty)}e^{-\lambda t}(f(X^{x}_t)\,dt +g(X^{x}_t)\,d\|K^{x}\|_t)\nonumber\\
&\quad+E\sum_{0\leq t<\infty}e^{-\lambda t}
\int_0^{|\Delta K^x_t|}\bar g(X^x_t-s{\mathbf n}(X^x_t))\,ds,\quad x\in\BR^d,
\end{align}
where $\|K^x\|_t=|K^x_0|+|K^x|_t$ is the  variational norm of $K^x$ on $[0,t]$  and
\begin{equation}
\label{eq1.10}
\bar g(x)=g(x)-g(\Pi(x)),\quad x\in\BR^d.
\end{equation}
Slightly different but equivalent representations are given in Section \ref{sec4}.
Clearly (\ref{eq1.3}) reduces (\ref{eq1.2}) in case $L=\LL$.  We think that (\ref{eq1.3})  provides useful information on the nature of the viscosity solution of  (\ref{eq1.1}). As an illustration of the  utility  of (\ref{eq1.3}) we show that if $u_{\alpha}$, $\alpha\in(1,2)$, is the viscosity solution of (\ref{eq1.1}) with $\II=-(-\Delta)^{\alpha/2}$,  then
\[
\lim_{\alpha\rightarrow2^{-}}\sup_{x\in\bar D}|u_{\alpha}(x)-\tilde u(x)|=0,
\]
where $\tilde u$ is the viscosity solution of the Neumann problem
\begin{equation}
\label{eq1.9}
-\tilde\LL u+\lambda u=f\quad\mbox{in }D,\qquad
\frac{\partial u}{\partial{\mathbf n}}=-g\quad\mbox{on }\partial D,
\end{equation}
with local operator $\tilde\LL$ defined by (\ref{eq1.8}) but with the diffusion matrix $a$ replaced by the matrix $\tilde a=a+2I_d$, where $I_d$ is the identity matrix. As another application we show some stability results with respect to the convergence of the coefficients $\sigma,b$.

It is worth noting  that our proof that   $u$ defined by (\ref{eq1.3}) is the viscosity solution of (\ref{eq1.1}) is probabilistic and independent of the reuslts of \cite{BGJ}. It uses some  ideas  from \cite{PZ}.
We also give a probabilistic proof of the convergence of $\{u_n\}$ to $u$.
The proof of the convergence of $\{u_n\}$ to $u$ given in \cite{BGJ} is analytic. However,
the penalization term used in \cite{BGJ} to define $u_n$ is an analytic counterpart to the penalization term used in known  approximation schemes for $X^x$  (see \cite{LS} for the case of diffusion process, i.e. when $L=\LL$,  and \cite{LS1,LS2} for the general case).
This simple observation is behind our method of proof.  We consider probabilistic solutions $u_n$ of problems
\begin{equation}
\label{eq1.4}
-L^nu_n+\lambda u_n=f +ng\cdot\mbox{dist}(\,\cdot,\bar D)\quad\mbox{in }\BR^d
\end{equation}
with
\[
L^n=L-n\sum^d_{i=1}(x_i-\Pi^i(x))\frac{\partial}{\partial x_i}.
\]
They are defined by
\begin{equation}
\label{eq1.5}
u_n(x)=E\int^{\infty}_0e^{-\lambda t}
\big(f(X^{x,n}_t)+ng(X^{x,n}_t)
\cdot\mbox{dist}(X^{x,n}_t,\bar D)\big)\,dt,\quad x\in\BR^d,
\end{equation}
where $X^{x,n}$ is the solution of the SDE in $\BR^d$ associated with $L^n$. By using rather standard arguments (see Section \ref{sec4}) one can see that $u_n$  is a viscosity solution of (\ref{eq1.4})  and $u$ is a viscosity solution of (\ref{eq1.1}). To show our main convergence result we assume that the L\'evy measure  $\nu$ satisfies the following integrability  condition: \begin{equation}
\label{eq1.6}
\int_{\{|y|\ge1\}}|y|^p\nu(dy)<\infty\quad \mbox{for some }p>1.
\end{equation}
We show that if (\ref{eq1.6}) is satisfied, $f,g$ are continuous, $|f(x)|\le K(1+|x|^p)$ for some $K\ge0$  and $g$ is bounded, then $u,u_n\in C(\BR^d)$ and for every compact $K\subset\BR^d$
\begin{equation}
\label{eq1.7}
\lim_{n\rightarrow\infty}\sup_{x\in K}|u_n(x)-u(x)|=0.
\end{equation}
 Condition (\ref{eq1.6}) may be omitted if $g=0$ and $f$ is bounded. It is worth adding that unlike \cite{BGJ} we do not assume that $g$ is Lipschitz continuous.

In case $L=\LL$ results of the form (\ref{eq1.7}) for parabolic semilinear equations were proved in \cite{BMZ}. The main problem in proving (\ref{eq1.7}) for nonlocal operators lies in the fact that in general $\{X^{x,n}\}$ does not converge in the Skorokhod $J_1$-topology to the first component $X^x$ of the solution of the (slightly generalized) Skorokhod equation associated with $L$.
Nevertheless, one can show the convergence of functionals of the form appearing on the right-hand side of (\ref{eq1.5}) to functionals on the right-hand side of (\ref{eq1.3}). In fact, we prove the convergence of such functionals in a much more general setting. We also give new estimates on the variation of the process $K^x$. We think that these results are of independent interest.

Finally, let us stress, that in \cite{BGJ} broader class of equations is treated. Furthermore, unlike \cite{BGJ}, in the present paper we only treat the case of bounded, convex, regular domain and normal reflection. We think that possible extension to more general operator and/or domains and oblique reflection deserves further study. Another problem is to extend our results to semilinear equations.

\section{Penalization and convergence of functionals}
\label{sec2}

In the paper, $D$ is a bounded convex $C^2$-domain
in $\BR^d$. For $x\in\BR^d$ we denote by $\Pi(x)=(\Pi^1(x),\dots,\Pi^d(x))$ the unique element $y\in\bar D$ such that $|y-x|=\mbox{dist}(x,\bar D)$. We denote by ${\mathbf n}(x)=({\mathbf n}_1(x),\dots,{\mathbf n}_d(x))$ the normal inward unit vector at
$x\in\partial D=\bar D\setminus D$. It is known that  the function
$\BR^d\setminus\bar D\ni x\mapsto\mbox{dist}(x,\bar D)$ is differentiable and
\[
\nabla(\mbox{dist}^2(x,\bar D))= 2(x-\Pi(x)),\quad x\in\BR^d\setminus\bar D.
\]
As in \cite{BGJ} we set
\[
\bar{\mathbf n}(x)=-\nabla\mbox{dist}(x,\bar D),\quad x\in\BR^d\setminus\bar D.
\]
Since
\begin{equation}
\label{eq2.1}
x-\Pi(x)=\mbox{dist}(x,\bar D)\cdot\nabla\mbox{dist}(x,\bar D)
=-\mbox{dist}(x,\bar D)\cdot\bar {\mathbf n}(x),
\end{equation}
we have
\begin{equation}
\label{eq2.2}
\bar{\mathbf n}(x)=\frac{\Pi(x)-x}{\mbox{dist}(x,\bar D)}=\frac{\Pi(x)-x}{|\Pi(x)-x|}\in\NN_{\Pi(x)},\quad\mbox{i.e.}\quad
\bar{\mathbf n}(x)={\mathbf n}(\Pi(x))
\end{equation}
(see \cite[Remark 1(ii)]{Sl}).

We set $\BR_+=[0,\infty)$ and denote by $\BD(\BR_+;\BD^d)$ the space of  $\BR^d$-valued functions on $\BR_+$ which are right continuous and have left-hand limits.
For a function  $k=(k^1,\dots, k^d)$ on $\BR_+$ of  locally finite variation and $0\leq a <b$ we denote by $|k|_{(a,b]}$ its total variation  on $[a,b]$, that is $|k|_{(a,b]}=\sup\sum_{i=1}^n|k_{t_i}-k_{t_{i-1}}|$, the supremum being taken over all partitions $a=t_0<t_1<\dots<t_n=b$. We also set $|k|_{(a,b)}=|k|_{(a,b]} -\Delta k_b$, $|k|_t=|k|_{(0,t]}$  and  $\|k\|_t=|k_0|+|k|_t$, $t\ge0$.

In the paper, the integral of an integrable function $h$ with respect to $k$ on the interval $(0,t]$  is denoted by $\int_0^t h_s\,dk_s$, and the integral on $[0,t]$ is denoted by $\int_{[0,t]}h_s\,dk_s$. In the second case, we adopt the convention that $k_{0-}=0$ (or, equivalently, that $\Delta k_0:=k_0-k_{0-}=k_0$). It follows that
\[
\int_{[0,t]}h_s\,dk_s=\int_0^t h_s\,dk_s+h_0k_0.
\]

\subsection{Deterministic case}
\label{sec2.1}

We start with the definition of the (slightly generalized) deterministic Skorokhod problem.

\begin{definition}
\label{def2.1}
Let $y\in\BD(\BR_+;\BR^d)$ be a function such that $y_0\in \BR^d$. A pair $(x,k)\in\BD(\BR_+;\BR^{2d})$ is called a solution of the Skorokhod problem associated with $y$ if
\begin{enumerate}
\item[(a)] $x_t=y_t+k_t$, $t\ge0$,

\item[(b)] $x$ is $\bar D$-valued, $k$ is a function of locally bounded variation such that $k_0=\Pi(y_0)-y_0$ and  $k_t=k_0+\int^t_0\mathbf{n}(x_s)\,d|k|_s$, $|k|_t=\int^t_0\fch_{\{x_s\in\partial D\}}\,d|k|_s$, $t\ge0$.
    \end{enumerate}
\end{definition}

\begin{remark}
\label{rem2.2} (i) The above definition is a direct extension of the ``usual" definition of a solution  given  by Skorokhod \cite{Sk1,Sk2} (see also  \cite{LS,Ta}). In the ``usual" definition it is assumed that $y_0\in\bar D$. In case $d=1$ and $y_0\in\BR$ our definition coincides with the definition considered in \cite{BKR}.
\smallskip\\
(ii) Note that $(x,k)$ is a solution of the Skorokhod problem in the sense of Definition \ref{def2.1} if and only if  $(x,k)=(\bar x, \Pi(y_0)-y_0+\bar k)$, where
$(\bar x,\bar k)$  is the solution of the usual Skorokhod problem
\[
\bar x_t=\bar y_t+\bar k_t, \quad t\ge0,
\]
associated with $\bar y_t=\Pi(y_0)-y_0+y_t$, $t\ge0$.
\smallskip\\
(iii) Suppose that condition (a) of Definition \ref{def2.1} is satisfied. Then condition (b) is equivalent to the following condition:
\begin{enumerate}
\item[(b')]
$x$ is $\bar D$-valued, $k$ is a function of locally bounded variation such that $k_{0-}=0$ and  $k_t=\int_{[0,t]}\mathbf{n}(x_s)\,d\|k\|_s$, $\|k\|_t=\int_{[0,t]}\fch_{\{x_s\in\partial D\}}\,d\|k\|_s$, $t\ge0$.
\end{enumerate}
\end{remark}
Let  $\{y^n\}\, \subset\BD(\BR_+;\BR^d)$. Consider the penalization scheme:
\begin{equation}
\label{eq2.3}
x^n_t=y^n_t-n\int^t_0(x^n_s-\Pi(x^n_s))\,ds=:y^n_t+k^n_t,\quad t\ge0.
\end{equation}

Note that the pair $(\Pi(x^n),k^n-y^n_0+\Pi(y^n_0)$ is a solution of the Skorokhod problem for $\Pi(x^n)-x^n+y^n+y^n_0-\Pi(y^n_0)$. This follows from the fact that
$\Pi(x^n)=\Pi(x^n)-x^n+y^n+k^n$.

Let $\omega'_x(\delta,q)$ denote the  modulus of continuity of $x\in\BD(\BR_+;\BR^d)$ on $[0,q]$ defined by $\omega'(\delta,q)=\inf\{\max_{i\le r} \omega_x([t_{i-1},t_i)):0=t_0<\dots< t_r=q,\inf_{i<r}(t_i-t_{i-1})\ge\delta\}$, where $\omega_x(I)=\sup_{s,t\in I}|x_s-x_t|$.  For  fixed $n\ge1$ let $T,\delta>0$ and $a\in D$  be chosen so that   $\omega'_{y^n}(\delta,T)<\mbox{\rm dist}(a,\partial D)/2$.
In \cite[Lemma 2.2]{LS2} it is proved that then

\begin{equation}
\label{eq2.4}
\sup_{t\le T}|x^n_t-a|\le 2\sqrt{7}([T/\delta]+1)\cdot\sup_{t\le T}|y^n_t-a|
\end{equation}
and
\begin{equation}
\label{eq2.5}
|k^n|_T\le 55([T/\delta]+1)^3(\mbox{\rm dist}(a,\partial D))^{-1}\sup_{t\le T}
|y^n_t-a|.
\end{equation}

\begin{theorem}
\label{th2.3}
Assume that  $y^n\to y$ in $J_1$. Then
\begin{enumerate}[\rm(i)]
\item If $y_0\in\bar D$ then $x_0=y_0$ and $x^n_0\rightarrow x_{0}$. Moreover, if $\Delta y_t=0$ then $x^n_t\rightarrow x_{t}$, $t>0$ (if $\Delta y_t=0$, then $x_{t-}=x_t$).

\item If $\{t_n\}$ is a sequence such that $t_n\to t$ and  $\Delta y^n_{t_n}\to \Delta y_t$, then $x^n_{t_n}\to x_{t-}+\Delta y_t$ and for any sequences $\{t_{n'}\}, \{t_{n''}\}$ such that $t_{n'}<t_n<t_{n''}$ and $t_{n'}\to t$,  $t_{n''}\to t$ we have  $x^n_{t_{n'}}\to x_{t-}$, $x^n_{t_{n''}}\to x_{t}$.
\item $(\Pi(x^n_t),y^n)\to (x,y)$ in $J_1$.
\end{enumerate}
\end{theorem}
\begin{proof}
See \cite[Theorem 2.3]{LS2} and \cite[Corollary 2.4]{LS2}.
\end{proof}
\begin{corollary}
\label {cor2.4}
Assume that $\varphi\in C(\BR^d)$, $f\in C(\BR_+\times\BR^d)$   and $y^n\to y$ in $J_1$-topology. Then
\begin{enumerate}[\rm(i)]
\item $\varphi(x^n_T)\rightarrow\varphi (x_T)$ provided that $\Delta y_T=0$,

\item $\int^T_0f(s,x^n_s)\,ds\rightarrow \int^T_0f(s,x_s)\,ds$.
\end{enumerate}
\end{corollary}
\begin{proof}
Assertion (i) follows from Theorem \ref{th2.3}(i). By (\ref{eq2.4}) there is $C>0$ such that
$\sup_{n\ge1}\sup_{t\le T}|x^n_t|\le C$. Therefore in the proof of (ii) we may and will assume that $f$ is bounded. Since the set $J=\{t\le T:|\Delta y_t|>0\}$ is countable and $f(s,x^n_s)\rightarrow f(s,x_s)$ for $s\in[0,T]\setminus J$ by Theorem \ref{th2.3}(i), applying the dominated convergence theorem we get (ii).
\end{proof}

In the following example we show that natural  functionals of integral type associated with $(x^n,k^n)$ in general do not converge to the respective functionals of $(x,k)$.

\begin{example}
\label{ex2.5}
Let $T,z>0$ and $a\in[0,T)$.  Consider the Skorokhod problem in $D=(0,1)$ associated with $y_t=-z{\bf 1}_{\{t\geq a\}}$, $t\geq0$. Note that in  case  $a=0$ we have $y_0=-z\notin\bar D$. It is easy to see that  the solution $(x,k)$ is of the form $x_t=0$, $k_t=z{\bf 1}_{\{t\geq a\}}$ for $t\geq0$. Moreover, the solution of (\ref{eq2.3}) with $y^n:=y$ has the form
\[
x^n_t=\begin{cases}
0, &t<a, \\
-ze^{-n(t-a)}, & t\ge a,
\end{cases}
\qquad
k^n_t=\begin{cases}
0, &t<a, \\
-n\int^t_ax^n_s\,ds=z( 1-e^{-n(t-a)}), & t\ge a.
\end{cases}
\]
Hence, for every $g\in C(\BR)$,
\begin{align*}
\int_0^Tg(x^n_s)\,d|k^n|_s&=\int_a^Tg(-ze^{-n(s-a)})nze^{-n(s-a)}\,ds\\
&=-\int^{ze^{-n(T-a)}}_{z}g(-r)dr\rightarrow\int^z_{0}g(-r)dr
\end{align*}
as $n\rightarrow\infty$. On the other hand,
\[\int_{[0,T]}g(x_s)\,d\|k\|_s=g(0)|k|_T+g(0)\cdot|\Pi(y_0)-y_0|=g(0)z,\]
because $|k|_T=0$ and $|\Pi(y_0)-y_0|=z$ in case $a=0$, and $|k|_T=z$, $|\Pi(y_0)-y_0|=0$ in case $a>0$.
\end{example}

Let $(x,k)$ be the solution of the Skorokhod problem associated with $y\in\BD(\BR_+;\BR^d)$
such that $y_0\in \BR^d$, and let
$g\in C(\BR_+\times\BR^d)$  and $T>0$.  In what follows we consider the following functional
\begin{equation*}
I_T(g,x,k):=\int_{[0,T]}g(s,x_s)\,d\|k\|_s+\sum_{0\leq s\leq T}\int_0^{|\Delta k_s|}\bar g(s,x_s-r{\mathbf n}(x_s))\,dr,
\end{equation*}
where $\bar g(s,x)=g(s,x)-g(s,\Pi(x))$, $x\in\BR^d$, $s\leq T$.
Note that for $g$  and  $(x,k)$  from Example \ref{ex2.5} we have
\[
I_T(g,x,k)=g(0)z+\int_0^{|\Delta k_a|}\bar g(x_a-r{\mathbf n}(x_a))\,dr
=g(0)z+\int_0^z\bar g(-r)\,dr=\int_{0}^z g(-r)\,dr.\]

The following Theorem \ref{th2.4}  and its stochastic version (Theorem \ref{th2.9} below) will play the  key role in proofs of our main results on the Neumann problem.

\begin{theorem}
\label{th2.4}
Assume that  $g\in C(\BR_+\times\BR^d)$   and
$y^n\to y$ in $J_1$. If $\Delta y_T=0$, then
 \[\int^T_0g(s,x^n_s)\,d|k^n|_s \rightarrow I_T(g,x,k).\]
\end{theorem}
\begin{proof}
{\em Step 1}. We will  show that
\begin{equation}
\label{eq2.11}
\int^T_0g(s,\Pi(x^n_s))\,d|k^n|_s \rightarrow\int_{[0,T]}g(s,x_s)\,d\|k\|_s.
\end{equation}
Note that under our assumption on the domain $D$ we have $d\|k\|_t=\varphi(x_t)\,dk_t$ and 
\begin{equation}
\label{eq2.24}
d|k^n|_t=\varphi(\Pi(x^n_t))\,dk^n_t
\end{equation} 
for some continuous $\varphi:\BR^d\rightarrow\BR^d$.
Therefore to get (\ref{eq2.11}) it suffices to show that if $\Delta y_T=0$, then
for any continuous $h:\BR_+\times\BR^d\rightarrow\BR^d$ we have
\begin{equation}
\label{eq3.7}
\int^T_0h(s,\Pi (x^n_s))\,dk^n_s
\rightarrow\int_{[0,T]}h(s,x_s)\,dk_s.
\end{equation}
Let $\{\delta^i\}$ be a sequence of strictly positive constants such that
$\delta^i \downarrow 0$ and $|\Delta y_t| \neq \delta ^i , t \in\Rp $. Similarly
to the proof of  \cite[Theorem 2.3]{LS2} we set $t^i_{n,0}= 0$,
$ t^i_{n,k+1}=\min ( t^i_{n,k} +
\delta _k^i, \inf\{t> t^i_{n,k}:|\Delta y^n_t|> \delta ^i \})$,
 $t^i_{0}=0$, $t^i_{k+1}=\min(t^i_{k} + \delta _k^i, \inf \{t>
t^i_{k}:|\Delta y_t|> \delta ^i \})$,  where $\{\{ \delta ^i_k\}\}$ is an  array
of constants satisfying $\delta^i/2 \leq
\delta _k^i \leq \delta ^i$ and $|\Delta y_{t^i_k + \delta
_k^i}|=0$. Since $\Delta y_T=0$,  without loss of generality  we  may assume that $T\notin\{\{t^i_k\}\}$.  Next, for each  $i\in\BN$ set
$ y^{n,(i)}_t = y^n_{t^i_{n,k}}$,  $t \in
[t^i_{n,k},t^i_{n,k+1})$,  $ y^{(i)}_t=y_{t^i_{k}}$, $t \in
[t^i_{k},t^i_{k+1})$, $n, k \in \N\cup\{0\}$.
Observe that for any $i\in\BN, k\in\BN\cup\{0\}$,
\begin{equation}
\label{eq2.13}
t^i_{n,k}\rightarrow t^i_k\quad\mbox{and}\quad y^n_{t^i_{n,k}}
	\rightarrow y_{t^i_k}
\end{equation}
as $n\rightarrow\infty$, and if $i\rightarrow\infty$, then
\begin{equation}
\label{eq2.14}
\limsup_{i\rightarrow\infty}\sup_{t\le T}|y^{(i)}_t-y_t|=0
\quad\mbox{and}\quad
\lim_{i\rightarrow\infty}\limsup_{n\rightarrow\infty}|y^{n,(i)}_t-y^n_t|=0.
\end{equation}
For $n,i\in\BN$ we denote by  $x^{n,(i)}$  the solution of the
equation with the penalization term of the form
\[
x^{n,(i)}_t=y^{n,(i)}_t-n\int_0^t(x^{n,(i)}_s-\Pi(x^{n,(i)}_s))\,ds=y^{n,(i)}_t
+k^{n,(i)}_t ,\quad t \in \Rp,
\]
and by $(x^{(i)},k^{(i)})$ we denote the solution of the Skorokhod problem associated with $y^{(i)}$.
By (\ref{eq2.14}) and arguments from the proof of \cite[Theorem 2.3]{LS2},
\begin{equation}
\label{eq2.7}
\lim_{i\to\infty}\sup_{t\leq T}|x^{(i)}_t-x_t|=0\quad\mbox{\rm and}\quad \lim_{i\to\infty}\limsup_{n\to\infty}\sup_{t\leq T}|x^{n,(i)}_t-x^n_t|=0.
\end{equation}
To shorten notation, we  set
\[
z^n_s=h(s,\Pi(x^n_s)),\qquad z^{n,(i)}_s=h(t(n,i)_s,\Pi(x^{n,(i)}_s))
\]
and
\[
z_s=h(s,x_s),\qquad z^{(i)}_s=h(t(i)_s,x^{(i)}),
\]
where $ t(n,i)_s = t^i_{n,k}$,  $s \in
[t^i_{n,k},t^i_{n,k+1})$,  $ t(i)_s=t^i_{k}$, $s \in
[t^i_{k},t^i_{k+1})$, $k \in \N\cup\{0\}$, $n,i\in\N$.
To show (\ref{eq3.7}) we first observe that
\begin{equation}
\label{eq2.8}
\Big|\int^T_0z^{n,(i)}_s\,dk^n_s-\int^T_0z^{n}_s\,dk^n_s\Big|\le \sup_{s\le T} |z^{n,(i)}_s-z^{n}_s|\cdot|k^n|_T.
\end{equation}
We next show that for every $i$, if $\Delta y_T=0$, then
\begin{equation}
\label{eq2.9}
\int^T_0z^{n,(i)}_s\,dk^n_s\rightarrow\int_{[0,T]}z^{(i)}_s\,dk_s
\end{equation}
as $n\rightarrow\infty$. By (\ref{eq2.13}) and Theorem \ref{th2.3}(ii), for every $k\in\BN$, $k^n_{t_{n,k}^i}=x^n_{t_{n,k}^i}-y^n_{t_{n,k}^i}\rightarrow x_{{t_{k}^i}-}+\Delta y_{t_{k}^i}-y_{t_{k}^i}=x_{{t_{k}^i}-}-y_{{t_{k}^i}-}=k_{{t_{k}^i}-}$. Moreover, $k^n_T\rightarrow k_T$. Set $k_0(n,i)=\max\{k:t^{i}_{n,k}< T\}<\infty$. Since $k^n_0=0$ and $k^n$ is continuous,
\begin{align*}
\int^{T}_0z^{n,(i)}_s\,dk^n_s
&=\sum_{k=0}^{k_0(n,i)-1}\int_{[t^{i}_{n,k},t^{i}_{n,k+1})}z^{n,(i)}_s\,dk^n_s
+\int_{[t^{i}_{n,k_0(n,i)},T]}z^{n,(i)}_s\,dk^n_s\\
&=\sum^{k_0(n,i)-1}_{k=0}z^{n,(i)}_{t^{i}_{n,k}}
(k^n_{t^{i}_{n,k+1}}-k^n_{t^{i}_{n,k}})+z^{n,(i)}_{t^{i}_{n,k_0(n,i)}}
(k^n_{T}-k^n_{t^{i}_{n,k_0(n,i)}}).
\end{align*}
For a sufficiently large $n$, $k_0(n,i)=k_0(i)=\max\{k:t^{i}_k< T\}<\infty$, so
\begin{align*}
\int^T_0z^{n,(i)}_s\,dk^n_s&\rightarrow\sum^{k_0(i)-1}_{k=0}z^{(i)}_{t^{i}_{k}}
(k_{t^{i}_{k+1}-}-k_{t^{i}_{k}-})+z^{(i)}_{t^{i}_{k_0(i)}}
(k_{T}-k_{t^{i}_{k_0(i)}-})\\
&=\sum_{k=0}^{k_0(i)-1}\int_{[t^{i}_{k},t^{i}_{k+1})}z^{(i)}_s\,dk_s
+\int_{[t^{i}_{k_0(i)},T]}z^{i}_s\,dk_s=\int_{[0,T]}z^{(i)}_s\,dk_s
\end{align*}
(with the convention that $k_{0-}=0$). This shows (\ref{eq2.9}). Furthermore, we have
\begin{equation}
\label{eq2.10}
\Big|\int_{[0,T]}z^{(i)}_s\,dk_s-\int_{[0,T]}z_s\,dk_s\Big|\le\sup_{s\leq T}|z^{(i)}_s-z_s|\cdot\|k\|_T.
\end{equation}
Since from (\ref{eq2.5}) we know that  $\sup_{n\ge1}|k^n|_T<\infty$, we deduce from  (\ref{eq2.7})--(\ref{eq2.10}) that  (\ref{eq3.7}) is satisfied and the proof of (\ref{eq2.11})  is complete.\\
{\em Step 2.} We will show that
\begin{equation}
\label{eq2.12}
\int^T_0\bar g(s,x^n_s)\,d|k^n|_s \rightarrow\sum_{0\leq s\leq T}\int_0^{|\Delta k_s|}\bar g(s,x_s-r{\mathbf n}(x_s))\,dr
\end{equation}
as $n\rightarrow\infty$. Let $y^{n,(i)}$, $x^{n,(i)}$, $k^{n,(i)}$  and $y^{(i)}$, $x^{(i)}$, $k^{(i)}$ etc. be defined as in {\em Step 1}. By elementary calculations one  can check that $x^{n,(i)}$  is of the form
\[
 x^{n,(i)}_t=\begin{cases}
y^n_0, &t=0, \smallskip\\
\Pi (x^{n,(i)}_{t^i_{n,k}})+( x^{n,(i)}_{t^i_{n,k}}-\Pi (
x^{n,(i)}_{t^i_{n,k}}))e^{-n(t-t^i_{n,k})},  & t \in
 (t^i_{n,k},t^i_{n,k+1}),\,k \in \N \cup \{0\}, \smallskip\\
x^{n,(i)}_{(t^i_{n,k+1})-}+y^{n,(i)}_{t^i_{n,k+1}}-y^{n,(i)}_{t^i_{n,k}},
 & t=t_{n,k+1},\,k \in \N\cup \{0\}.
\end{cases}
\]
Fix $i\in \BN$.  Since for a sufficiently large $n$, $k_0(n,i)=k_0(i)$, we may and will assume that for such $n$, $t^{i}_{n,k_0(n,i)+1}=T$. Set $p_{n,k}^i=\Pi (
x^{n,(i)}_{t^i_{n,k}})$, $d_{n,k}^i=x^{n,(i)}_{t^i_{n,k}}-\Pi (
x^{n,(i)}_{t^i_{n,k}})$ and $\Delta t_{n,k+1}^i=t^{i}_{n,k+1}-t^{i}_{n,k}$,  and observe that for all sufficiently large $n$,
\begin{align*}&\int_0^T\bar g(t(n,i)_s,x^{n,(i)}_s)\,dk^{n,(i)}_s\\
&\quad=\sum_{k=0}^{k_0(n,i)}\int_{t^{i}_{n,k}}^{t^{i}_{n,k+1}}\bar g(t^i_{n,k},p_{n,k}^i+d_{n,k}^ie^{-n(s-t^i_{n,k})})n|d_{n,k}^i|e^{-n(s-t^i_{n,k})}\,ds
\\
&\quad=\sum_{k=0}^{k_0(i)}\int_{t^{i}_{n,k}}^{t^{i}_{n,k}+\Delta t_{n,k+1}^i}\bar g(t^i_{n,k},p_{n,k}^i+d_{n,k}^ie^{-n(s-t^i_{n,k})})n|d_{n,k}^i|e^{-n(s-t^i_{n,k})}\,ds
=\sum_{k=0}^{k_0(i)} I_{n,k}^i.
\end{align*}
Changing the variables $s\leadsto r:=|d_{n,k}^i|e^{-n(s-t^i_{n,k})}$,
for $k=0,\dots,k_0(i)$ we obtain
\begin{align*}
I_{n,k}^i&=\int_{t^{i}_{n,k}}^{t^{i}_{n,k}
+\Delta t_{n,k+1}^i}\bar g(t^i_{n,k},p_{n,k}^i-{\mathbf n}( p_{n,k}^i) |d_{n,k}^i|e^{-n(s-t^i_{n,k})})n|d_{n,k}^i|e^{-n(s-t^i_{n,k})}\,ds\\
&=-\int_{|d_{n,k}^i|}^{|d_{n,k}^i|e^{-n\Delta t_{n,k+1}^i}}
\bar g(t^{i}_{n,k},p_{n,k}^i-r{\mathbf n}( p_{n,k}^i))\,dr.
\end{align*}
By Theorem \ref{th2.3}, $p_{n,k}^i\rightarrow x^{(i)}_{t_k}$  and
$d_{n,k}^i\rightarrow x^{(i)}_{t_{k-1}}+\Delta y^{(i)}_{t_k}-\Pi(x^{(i)}_{t_{k-1}}+\Delta y^{(i)}_{t_k})=-\Delta k^{(i)}_{t_k}$.
From this and the fact that $\Delta t_{n,k+1}^i\rightarrow t^i_{k+1}-t_k^i>0$ we deduce that
\[
I_{n,k}^i\rightarrow\int_0^{|\Delta k^{(i)}_{t_k}|}
\bar g (t^i_k,x^{(i)}_{t_k}-r{\mathbf n}( x^{(i)}_{t^i_{k}}))\,dr.
\]
Since the set $\{t_k\}$ exhausts all the times of jumps of $y^{(i)}$ and hence of $k^{(i)}$,
we have proved that for every $i\in\BN$,
 \begin{equation}
\label{eq2.16}
\int^T_0\bar g(t(i)_s,x^{n,(i)}_s)\,d|k^{n,(i)}|_s \rightarrow\sum_{0\leq s\leq T}\int_0^{|\Delta k^{(i)}_s|}\bar g(t(i)_r,x^{(i)}_s-r{\mathbf n}(x^{(i)}_s))\,dr.
 \end{equation}
We are going to show that from (\ref{eq2.16}) one can deduce (\ref{eq2.12}). By (\ref{eq2.5}) and (\ref{eq2.7}) there exists a compact set $K\subset\BR^d$ such that $x^{(i)}_t, x^{n,(i)}_t, x^n_t\in K$ for $t\le T$. Furthermore, for $g\in C([0,T]\times\BR^d)$ there exists a sequence of functions  $\{g_j\}\subset C^2([0,T]\times\BR^d)$ such that for every $s\in[0,T]$,
\begin{equation}
\label{eq2.17}
\lim_{j\rightarrow\infty}\sup_{x\in K}|g(s,x)-g_j(s,x)|=0,
\end{equation}
and moreover, for every $j\in\N$  there is $L_j>0$  such that
\begin{equation}
\label{eq2.18}|
g_j(s,x)-g_j(s,y)|\leq L_j|x-y|,\quad x,y\in K.
\end{equation}
Fix $j\in\N$ and write $\bar g_j(s,x)=g_j(s,x)-g_j(s,\Pi(x))$ for $s\in[0,T]$,  $x\in\BR^d$. Clearly, (\ref{eq2.17}) and (\ref{eq2.18}) with the constant $2L_j$ hold true for $\bar g_j$.  By  (\ref{eq2.4}), without loss of  generality we may and will assume that the values of $x^n_s$, $x^{n,(i)}_s$, $x_s$  and $x^{(i)}_s$ belong to some compact set $K$.
We know that $\sup_{0\leq s\leq T}|k^{(i)}_s-k_s|\rightarrow0$. Let $\varepsilon>0$  be such that $|\Delta k_s|\neq \varepsilon$, $s\in[0,T]$. Note that for every $s$ such that $|\Delta k_s|>\varepsilon$ we have  $\Delta k^{(i)}_{s}\to\Delta k_{s}$. Therefore  it follows from (\ref{eq2.7}) that
\begin{align*}
\sum_{0\leq s\leq T,|\Delta k^{(i)}_s|>\varepsilon}&\int_0^{|\Delta k^{(i)}_s|}\bar g_j(t(i)_s,x^{(i)}_s-r{\mathbf n}(x^{(i)}_s))\,dr\\
&\qquad\quad\rightarrow
\sum_{0\leq s\leq T,|\Delta k_s|>\varepsilon}\int_0^{|\Delta k_s|}
\bar g_j(s,x_s-r{\mathbf n}(x_s))\,dr
\end{align*}
as $i\rightarrow\infty$. On the other hand, by (\ref{eq2.18}),
\begin{align*} \sum_{0\leq s\leq T,|\Delta k^{(i)}_s|\leq\varepsilon}&\int_0^{|\Delta k^{(i)}_s|}|\bar g_j(t(i)_s,x^{(i)}_s-r{\mathbf n}(x^{(i)}_s))|\,dr\\
&\leq \sum_{0\leq s\leq T,|\Delta k^{(i)}_s|\leq\varepsilon}2L_j\frac{|\Delta k^{(i)}_s|^2}{2}\leq \varepsilon L_j\sum_{0\leq s\leq T}|\Delta k^{(i)}_s|
\end{align*}
and
\[
\sum_{0\leq s\leq T,|\Delta k_s|\leq\varepsilon}\int_0^{|\Delta k_s|}
|\bar g_j(s,x_s-r{\mathbf n}(x_s)|\,dr\leq\varepsilon L_j\sum_{0\leq s\leq T}|\Delta k_s|.
\]
Since $\sup_{i\ge1}\|k^{(i)}\|_T$, $\|k\|_T<\infty$, letting $\varepsilon\downarrow0$  yields
\begin{equation}
\label{eq2.19}
\sum_{0\leq s\leq T}\int_0^{|\Delta k^{(i)}_s|}
\bar g_j(t(i)_s,x^{(i)}_s-r{\mathbf n}(x^{(i)}_s))\,dr\rightarrow
\sum_{0\leq s\leq T}\int_0^{|\Delta k_s|}\bar g_j(s,x_s-r{\mathbf n}(x_s))\,dr.
\end{equation}
We next observe that
\begin{align*}I^{n,(i)}&:=\int^T_0\bar g_j(s,x^{n,(i)}_s)\,d|k^{n,(i)}|_s -\int^T_0\bar g_j(s,x^n_s)\,d|k^n|_s \\
&=\int^T_0(\bar g_j(s,x^{n,(i)}_s)-\bar g_j(s,x^n_s))\,d|k^{n,(i)}|_s\\
&\quad+\int^T_0\bar g_j(s,x^{n}_s)n(|x^{n,(i)}_s-\Pi(x^{n,(i)}_s)|-|x^{n}_s-\Pi(x^{n}_s)|)\,ds
=:I_1^{n,(i)}+I_2^{n,(i)}.
\end{align*}
By the Lipschitz continuity of $\bar g_j$,
\[
|I_1^{n,(i)}|\leq 2L_j\sup_{t\leq T}|x^{n,(i)}_t-x^n_t|\cdot|k^{n,i}|_T,
\]
and by the Lipschitz continuity of $g_j$,
\[|I_2^{n,(i)}|\leq 2\sup_{t\leq T}|x^{n,(i)}_t-x^n_t|\cdot L_j |k^{n}|_T.\]
Since  $\sup_{n,i}|k^{n,(i)}|_T<\infty$ and $\sup_{n}|k^{n}|_T<\infty$ by (\ref{eq2.5}), it follows by (\ref{eq2.7}) that
\begin{equation}\label{eq2.20}
\lim_{i\to\infty}\limsup_{n\to\infty}|I^{n,(i)}|=0.\end{equation}
Putting together (\ref{eq2.16}), (\ref{eq2.19})  and (\ref{eq2.20}) we obtain (\ref{eq2.12})  with $\bar g$ replaced by $\bar g_j$. By (\ref{eq2.17})  and the fact that
$\sup_{n\ge1} |k^n|_T<\infty$ and $|k|_T<\infty$  we get (\ref{eq2.12}) in the general case.  This and (\ref{eq2.11}) completes the proof of the theorem.
\end{proof}


\begin{theorem}
\label{th2.5} Assume that $y^n\to y$ in $J_1$-topology.  Let $(x^n,k^n)$ and  $(x,k)$ be the solutions of the  Skorokhod problem associated with $y^n$  and $y$, respectively.
For all $T>0$ and  $f_n,f,g_n,g\in C(\BR_+\times\BR^d)$ such that $\Delta y_T=0$, if
\begin{equation}
\label{eq2.21}
\sup_{0\leq s\leq T,\,x\in\bar D}|f_n(s,x)-f(s,x)|\rightarrow0,\qquad\sup_{0\leq s\leq T,\,x\in\bar D}|g_n(s,x)-g(s,x)|\rightarrow0
\end{equation}
as $n\rightarrow\infty$, then
\begin{equation}
\label{eq2.22}
\int_0^Tf_n(s,x^n_s)\,ds\rightarrow\int_0^T f(s,x_s)\,ds,\qquad I_T(g_n,x^n,k^n) \rightarrow I_T(g,x,k).
\end{equation}
\end{theorem}
\begin{proof} The first convergence in (\ref{eq2.22}) follows from the fact that $(x^n,k^n)\rightarrow(x,k)$ in $J_1$-topology.
Under our assumptions on the domain $D$, $d\|k\|_t=\varphi(x_t)\,dk_t$ and $d|k^n|_t=\varphi(x^n_t)\,dk^n_t$ for some  $\varphi\in C(\BR^d)$. Therefore using the classical result on the convergence of  integrals (see, e.g., \cite[Proposition 2.9]{JMP}) we get
\[
\int_{[0,T]} g_n(s,x^n_s)\,d\|k^n\|_s \rightarrow\int_{[0,T]} g(s,x_s)\,d\|k\|_s.
\]
On the other hand, using an approximation of the function $g$ by locally  Lipschitz continuous functions satisfying (\ref{eq2.17}), (\ref{eq2.18})  and  arguments from the proof of
(\ref{eq2.19})  we get
\begin{align*}
\sum_{0\leq s\leq T}\int_0^{|\Delta k^n_s|}
&(g_n(s,x^n_s-r{\mathbf n}(x^n_s))-g_n(s,x^n_s))\,dr\\
&\rightarrow\sum_{0\leq s\leq T}\int_0^{|\Delta k_s|}(g(s,x_s-r{\mathbf n}(x_s))
-g(s,x_s))\,dr,
\end{align*}
which completes the proof.
\end{proof}

\subsection{Stochastic case}

We assume as given a filtered probability space $(\Omega,\FF, (\FF_t)_{t\ge0} ,P)$ satisfying the usual condition. All processes considered below are assumed to be defined on this space and have sample paths in $\BD(\BR_+;\BR^d)$.

\begin{definition}
Let $Y$ be an $(\FF_t)$-adapted process with initial value $Y_0$.
A pair $(X,K)$ of $(\FF_t)$-adapted processes is called a solution of the Skorokhod problem associated with $Y$ if
\begin{enumerate}
\item[(a)] $X_t=Y_t+K_t$, $t\ge0$,

\item[(b)] $X$ is $\bar D$-valued, $K$ is a process of locally bounded variation such that $K_0=\Pi(Y_0)-Y_0$, $K_t=K_0+\int^t_0\mathbf{n}(X_s)\,d|K|_s$ and $|K|_t=\int^t_0\fch_{\{X_s\in\partial D\}}\,d|K|_s$, $t\ge0$.
    \end{enumerate}
\end{definition}
Let  $\{Y^n\}$ be a sequence of processes. Consider the following equations with the penalization terms:
\begin{equation}
\label{eq2.23}
X^n_t=Y^n_t-n\int^t_0(X^n_s-\Pi(X^n_s))\,ds=:Y^n_t+K^n_t,\quad t\ge0.
\end{equation}

The following two theorems are immediate consequences of the results of Section \ref{sec2.1}.

\begin{theorem}
\label{th2.9}
Let $(X,K)$ be the solution of the Skorokhod problem associated with  $Y$ and $(X^n,K^n)$, $n\ge1$, be the solution of \mbox{\rm(\ref{eq2.23})}. Assume that $\varphi:\BR^d\rightarrow\BR$ and $f,g:\BR_+\times\BR^d\rightarrow\BR$ are continuous, $g$ is bounded and there exist $C,p>0$ such that
\begin{equation}
\label{eq2.01}
|f(s,x)|\leq K(1+|x|),\quad s\geq0,\,x\in\BR^d.
\end{equation}
Moreover, assume that $Y^n\rightarrow Y$ weakly in the space $\BD(\BR_+;\BR^d)$ equipped with the   $J_1$-topology. Then
\begin{enumerate}[\rm(i)]
\item If $P(\Delta Y_T=0)=1$, then $E\varphi(X^n_T)\rightarrow E\varphi(X_T)$.

\item If the sequence $\{\sup_{s\leq T}|X^n_s|^p\}_{n\ge1}$ is uniformly integrable, then \[E\int^T_0f(s,X^n_s)\,ds\rightarrow E\int^T_0f(s,X_s)\,ds.\]

\item If $P(\Delta Y_T=0)=1$ and  the sequence $\{|K^n|_T\}_{n\ge1}$ is uniformly integrable, then \[E\int^T_0g(s,X^n_s)\,d|K^n|_s\rightarrow EI_T(g,X,K).\]
\end{enumerate}
\end{theorem}
\begin{proof} By the Skorokhod representation theorem we may assume that
$Y^n\rightarrow Y$
$P$-a.s. in the Skorokhod topology $J_1$. Therefore  (i)--(iii)
follows immediately from Corollary \ref{cor2.4}, Theorem \ref{th2.4} and the Lebesgue dominated convergence  theorem.
\end{proof}

\begin{theorem}
\label{th2.10}
Let $(X,K)$ be the solution of the Skorokhod problem associated with  $Y$ and $(X^n,K^n)$, $n\ge1$, be the solutions of the Skorokhod problem associated  with  a process $Y^n$. Assume that $f,f_n,g,g_n:\BR_+\times\BR^d\rightarrow\BR$ are continuous and satisfy (\ref{eq2.21}).
If  $Y^n\rightarrow Y$ weakly in the space $\BD(\BR_+;\BR^d)$ equipped with the   $J_1$-topology, then
\begin{enumerate}[\rm(i)]
\item $E\int^T_0f_n(s,X^n_s)\,ds\rightarrow E\int^T_0f(s,X_s)\,ds$.
\item If $P(\Delta Y_T=0)=1$ and  the sequence $\{|K^n|_T\}_{n\ge1}$ is uniformly integrable, then
    \[
    EI_T(g_n,X^n,K^n)\rightarrow EI_T(g,X,K).
    \]
\end{enumerate}
\end{theorem}
\begin{proof} Follows from the  Skorokhod representation theorem and Theorem \ref{th2.5}.
\end{proof}

\section{SDEs and reflected SDEs}

For the convenience of the reader, we start with recalling some known results on SDEs and reflected SDEs. They will be used in the next section. We will need the following assumptions.
\begin{enumerate}
\item[\rm(A1)]$\sigma:\BR^d\rightarrow\BR^d\times\BR^d$, $b:\BR^d\rightarrow\BR^d$ are bounded Lipschitz continuous functions.

\item[\rm(A2)] $\nu$ is a measure on $\BR^d$ such that $\nu(\{0\})=0$ and $\int_{\BR^d}(1\wedge|y|^2)\,\nu(dy)<\infty$ (i.e. $\nu$ is a L\'evy measure).

\item[\rm(A3)] $\int_{\{|y|>1\}}|y|^{p}\,\nu(dy)<\infty$ for some $p>1$.
\end{enumerate}

We first recall the definition of a solution of the reflected SDE. It is convenient to give it for general equations driven by an arbitrary semimartingale.

\begin{definition}
(i) Let $(\Omega,\FF,P)$ be a probability space and $(\FF_t)$ be a filtration on it satisfying the usual conditions. Furthermore, let $V$ be an $\FF_0$-measurable random variable with values in $\BR^d$, $Z$ be a $d$-dimensional $(\FF_t)$-semimartingale with $Z_0=0$ and let $\Phi:\BR^d\rightarrow\BR^m$. A pair $(X,K)$ of $(\FF_t)$-adapted processes is a strong solution of the reflecting SDE
\begin{equation}
\label{eq3.5}
X_t=V+\int^t_0\Phi(X_{s-})\,dZ_s+K_t,\quad t\ge0,
\end{equation}
if $(X,K)$ is a solution of the Skorokhod problem associated with $Y$ defined as $Y_t=V+\int^t_0\Phi(X_{s-})\,dZ_s$, $t\ge0$. \smallskip\\
(ii) We say that (pathwise) uniqueness hold for (\ref{eq3.5}) if for any two $(\FF_t)$-adapted solutions  $(X,K)$, $(X',K')$ we have $P((X_t,K_t)=(X'_t,K'_t),t\ge0)=1$.
\end{definition}

\begin{remark}
(i) By \cite[Theorem 5]{Sl}, if $\Phi$ is Lipschitz and bounded, then there exists a strongly  unique strong solution of the equation
\[
\bar X_t=\Pi(V)+\int^t_0\Phi(\bar X_{s-})\,dZ_s+\bar K_t,\quad t\ge0.
\]
(ii) By Remark \ref{rem2.2}(ii), the pair $(X,K)$ defined as $X=\bar X$, $K=\Pi(V)-V+\bar K$, is a unique strong solution of (\ref{eq3.5}).
\end{remark}

Consider now the reflected SDE on $D$ of the form
\begin{align}
\label{eq3.1}
X^{x,i}_t&=x_i+\sum^d_{j=1}\int^t_0\sigma_{ij}(X^{x}_r)\,dW^j_s
+\int^t_0b_i(X^{x}_s)\,ds +N^i_t+K^{x,i}_t\nonumber\\
&=Y^{x,i}_t+K^{x,i}_t, \quad t\ge 0,
\end{align}
for $i=1,\dots,d$, where $x=(x_1,\dots,x_d)\in\BR^d$, $W=(W^1,\dots,W^d)$ is a standard Wiener process and $N=(N^1,\dots,N^d)$  is a pure-jump $d$-dimensional L\'evy proces independent of $W$ starting from 0 with L\'evy measure $\nu$. Note that  $N$ has  the L\'evy--It\^o decomposition
of the form
\begin{align}
\label{eq3.01}
N_t&=bt+\int_{\{|x|\leq 1\}}x\,(\mu(t,dx)-t\nu(dx))+\int_{\{|x|>1\}}x\,\mu(t,dx)\nonumber\\
&=N^{(1)}_t+N^{(2)}_t+N^{(3)}_t,\quad t\geq0,
\end{align}
where $b\in\BR^d$  and $\mu$ denotes the  measure of jumps of $N$ (see, e.g.,
\cite[page 116]{A}). Equation (\ref{eq3.1}) is a special case of (\ref{eq3.5}), because
it  can be rewritten as
\begin{equation}
\label{eq3.6}
X^{x,i}_t=x_i+
\sum^{d+1+d}_{j=1}\int^t_0\Phi^i_j(X^x_{s-})\,dZ^{j}_s+K^{x,i}_t,
\quad t\ge0,\quad i=1,\dots,d,
\end{equation}
where
$Z_t=(W_t,t,N_t)$, $t\ge0$, and
$\Phi^i_j(x)=\sigma_{ij}(x)$, $\Phi^i_{d+1}(x)=b_i(x)$, $\Phi^i_{d+1+i}(x)=1$, $\Phi^i_{d+1+j}=0$ for $j=d+1,\dots,d+i,d+i+2,\dots, d+1+d$. Therefore, if $\sigma,b,\nu$ satisfy (A1), (A2), then there exists a unique strong solution $(X^x,K^x)$ of (\ref{eq3.1}). In fact,
\begin{equation}
\label{eq2.15}
X^x_t=X^{\Pi(x)}_t,\qquad K^x_t=\Pi(x)-x+K^{\Pi(x)}_t,\quad t\ge0,
\end{equation}
where $(X^{\Pi(x)},K^{\Pi(x)})$ is the unique strong solution of the (usual) Skorokhod equation of the form (\ref{eq3.6}) with initial condition $\Pi(x)\in\bar D$.

Define $Z$, $\Phi$ as in (\ref{eq3.6}). Since the function $x\mapsto\Pi(x)$ is Lipschitz continuous, by well known results (see, e.g., \cite[Theorem V.7]{Pr}) there exists a (pathwise) unique strong solution $X^{x,n}=(X^{x,n,1},\dots,X^{x,n,d})$ of the SDE
\[
X^{x,n,i}_t=x_i+\sum^d_{j=1}\int^t_0\Phi^i_j(X^{x,n}_{s-})\,dZ^j_s
-n\int^t_0(X^{x,n,i}_s-\Pi^i(X^{x,n}_s))\,ds.
\]
In different words,
\begin{align}
\label{eq3.2}
X^{x,n,i}_t&=x_i+\sum^d_{j=1}\int^t_0\sigma_{ij}(X^{x,n}_s)\,dW^j_s
+\int^t_0b_i(X^{x,n}_s)\,ds  +N^i_t+K^{x,n,i}_t \nonumber\\
&=Y^{x,n,i}_t+K^{x,n,i}_t,\quad t\ge0,
\end{align}
where
\begin{equation}
\label{eq3.13}
K^{x,n,i}_t=-n\int^t_0(X^{x,n,i}_s-\Pi^i(X^{x,n}_s))\,ds,\qquad Y^{x,n,i}_t=X^{x,n,i}_t-K^{x,n,i}_t.
\end{equation}

Summarizing, we have the following result.

\begin{proposition}
\label{prop3.3}
Let $x\in\BR^d$. If  \mbox{\rm(A1), (A2)} are satisfied, then there exits a (strongly) unique strong solution $(X^x,K^x)$ of \eqref{eq3.1} and for each $n\ge1$ there exists a unique strong solution $X^{x,n}$ of \eqref{eq3.2}.
\end{proposition}

Note also that the pathwise uniqueness of \eqref{eq3.2} implies that $X^{x,n}$   has the Markov property. For our purposes it is convenient to state it in the following form:
for every bounded Borel function $h$ on $\BR^d$:
\begin{equation}
\label{eq3.4}
E(h(X^{x,n}_{t+s})|\FF_t)=Eh(X^{y,n}_s)|_{y=X^{x,n}_t}\,,\quad s,t\ge0
\end{equation}
(see \cite[p. 301]{Pr}). Similarly, since the solution $(X^x,K^x)$ of (\ref{eq3.6}) is pathwise unique, for $v\ge t$ we have $X^x_v=X^{t,y}_v|_{y=X^{x}_t}$ and $K^x_v-K^x_{t}=K^{t,y}_v|_{y=X^x_t}$, where 
$(X^{t,x},K^{t,x})$ denotes the solution of (\ref{eq3.6}) starting at time $t$ from $x$. From this and the fact that the driver $Z$ in (\ref{eq3.6}) is a L\'evy process  we deduce that 	for any bounded Borel functions  $h,g$ on $\BR^d$,
\begin{equation}
\label{eq3.9}
E(h(X^{x}_{t+s_1})g(K^x_{t+s_2}-K^x_t)|\FF_t)
=Eh(X^{y}_{s_1})g(K^y_{s_2})|_{y=X^{x}_t}
\,,\quad s_1,s_2,t\ge0.
\end{equation}

The next result will play important role in the convergence results of Section \ref{sec4}.

\begin{theorem}
\label{th3.3} Let $x\in\BR^d$. If \mbox{\rm(A1)--(A3)} are satisfied, then for every $t\ge0$ the random variables $|K^{x}|^p_t$, $|K^{x,n}|^p_t$  are integrable. Moreover, there is a constant $C>0$ independent of $t,n$  such that
\[
\max\big(E|K^{x}|^p_t,E|K^{x,n}|^p_t\big)\leq C(1+c(t)),\quad t\ge0,\,n\ge1,
\]
where $c(t)=t^p$ for $t\geq1$  and $t^{(p/2)\wedge1}$ for $t\in[0,1)$.
\end{theorem}
\begin{proof} Since $D$ is a  convex $C^2$-domain
in $\BR^d$, there exists $\psi\in C^2(\BR^d)$ such that
\[
D=\{x\in\BR^d:\psi(x)>0\},\qquad \partial D=\{x\in\BR^d:\psi(x)=0\}
\]
and $\psi(x)=-\mbox{\rm dist}(x,\bar D)$, $x\in\BR^d\setminus \bar D$, $\nabla \psi(x)={\mathbf n}(x)$, $x\in\partial D$  (see, e.g., \cite[page 119]{GPP}). Note that in our case (we restrict our attention to bounded sets) the first  and second derivatives of $\psi$ are bounded. This is obvious  for $x\in\bar D$. For $x\in\BR^d\setminus\bar D$ this  follows from the  equalities
\[
\nabla\psi(x)=\frac{\Pi(x)-x}{|\Pi(x)-x|}={\mathbf n}(\Pi(x))=\nabla \psi(\Pi(x)).
\]
and boundedness of the derivatives  of $\nabla \psi $ and $\Pi$.
Let $(X^{x,n},K^{x,n})$ be a solution of \eqref{eq3.2}. By  It\^o's formula,  for $\psi (X^{x,n}_t)$ we have
\begin{align}
\label{eq3.14}
\psi(X^{x,n}_t)&=\psi(x)+\int_0^t\nabla\psi(X^{x,n}_{s-})\,dY^{x,n}_s
+\int_0^t\nabla\psi(X^{x,n}_{s-})\,dK^{x,n}_s\nonumber \\
&\quad+\frac12\sum_{i=1}^d\sum_{j=1}^d
\int_0^t\frac{\partial^2\psi}{\partial x_i\partial x_j}
(X^{x,n}_s)\,d[X^{x,n,i},X^{x,n,j}]^c_s \nonumber\\
&\quad +\sum_{s\leq t}(\psi(X^{x,n}_s)-\psi(X^{x,n}_{s-})-\nabla\psi(X^{x,n}_{s-})
\Delta X^{x,n}_s).
\end{align}
Since $K^{n,x}$ has  continuous sample paths and changes only when
$X^{n,x}\notin\bar D$,
\begin{align}
\label{eq3.15}
|K^{x,n}|_t&\le n\sqrt{d}\int^t_0|\Pi (X^{x,n}_s)-X^{x,n}_s|\,ds\nonumber\\
&=n\sqrt{d}\int^t_0|\Pi (X^{x,n}_s)-X^{x,n}_s|
\fch_{\{X^{n,x}_s\notin\bar D\}}\,ds\nonumber\\
&=-n\sqrt{d}\int^t_0\frac{(\Pi(X^{x,n}_s)-X^{x,n}_s)}{|\Pi(X^{x,n}_s)-X^{x,n}_s)|}
(X^{x,n}_s-\Pi(X^{x,n}_s)){\bf 1}_{\{X^{n,x}_s\notin\bar D\}}\,ds \nonumber\\
&=\sqrt{d}\int_0^t\nabla \psi(X^{x,n}_s)\,dK^{n,x}_s
=\sqrt{d}\int_0^t\nabla \psi(X^{x,n}_{s-})\,dK^{n,x}_s.
\end{align}
Set $c_0=\mbox{dist}(x,\bar D)+\sup_{x\in\bar D}\psi(x)$. Since $\psi(x)<0$ for $x\in\BR^d\setminus \bar D$  and the second derivatives $\frac{\partial^2\psi}{\partial x_i\partial x_j}$ and the coefficients $\sigma_{ij}$ are bounded  functions,  it follows from (\ref{eq3.14}) and (\ref{eq3.15}) that there is $c_1>0$ such that
\begin{align*}
\frac{1}{\sqrt{d}}|K^{x,n}|_t&\le c_0 +c_1t-\int_0^t\nabla\psi(X^{x,n}_{s-})\,dY^{x,n}_s\\
&\quad-\sum_{s\leq t}(\psi(X^{x,n}_s)-\psi(X^{x,n}_{s-})-\nabla\psi(X^{x,n}_{s-})\Delta X^{x,n}_s)\\
&\le c_0 +c_1t+\Big|\int_0^t\nabla\psi(X^{x,n}_{s-})\,d(Y^{x,n}_s-N^{(3)}_s)\Big|\\
&\quad+\Big|\sum_{s\leq t,\,|\Delta N_s|>1}(\psi(X^{x,n}_s)-\psi(X^{x,n}_{s-}))\Big|\\
&\quad+\Big|\sum_{s\leq t,\,|\Delta N_s|\leq1}(\psi(X^{x,n}_s)-\psi(X^{x,n}_{s-})-\nabla\psi(X^{x,n}_{s-})\Delta X^{x,n}_s)\Big|\\
&:=c_0 +c_1t+|I^{(1)}_t|+|I^{(2)}_t|+|I^{(3)}_t|.
\end{align*}
In what follows $c$ denotes a nonnegative constant which can vary from line to line but is independent of $t,n$. To estimate $E|I^{(1)}_t|^p$, we first observe that by using boundedness of $\nabla \psi$ and coefficients $b_i$, $\sigma_{ij}$, $i,j=1,...,d$, and making standard  calculations we obtain
\begin{equation}
\label{eq3.001}
E\Big|\int_0^t\nabla\psi(X^{x,n}_{s-})\,d(Y^{x,n}_s-N^{(2)}_s-N^{(3)}_s)\Big|^p\leq c\max(t^p,t^{p/2}).
\end{equation}
Moreover, by the Burkholder--Davis--Gundy inequality,
\[
E\Big|\int_0^t\nabla\psi(X^{x,n}_{s-})\,dN^{(2)}_s\Big|^p\le cE[N^{(2)}]_t^{p/2}=cE\Big(\sum_{s\leq t,\,|\Delta N_s|\leq1}|\Delta N_s|^2\Big)^{p/2}.
\]
Assume that $p\in(1,2)$. Since $|\Delta N_s|^2\fch_{\{\Delta N_s|\le 1\}}
\le|\Delta N_s|^2\wedge1$ and $p/2\in(1/2,1)$, we have
\begin{align*}
E\Big(\sum_{s\leq t,\,|\Delta N_s|\leq1}|\Delta N_s|^2\Big)^{p/2}
&\le \Big(E\sum_{s\leq t}|\Delta N_s|^2\wedge1\Big)^{p/2} \nonumber\\
&=\Big(t\int_{\BR^d}(|x|^2\wedge1)\,\nu(dx)\Big)^{p/2}=ct^{p/2}.
\end{align*} 
By this and (\ref{eq3.001}),
\begin{equation}
\label{eq3.003}
E|I^{(1)}_t|^p\le c\max(t^{p},t^{p/2}).
\end{equation}
In case $p\geq2$ we get (\ref{eq3.003}) by using  \cite[Theorem 4.4.23]{A} (Kunita's first
inequality).
To estimate $I^{(2)}_t$ and $I^{(3)}_t$ we use similar  arguments. We assume that $p\in(1,2)$ (the desired estimates in the case where $p\geq2$ follow from  Kunita's first inequality).
Since $|\psi(X^{x,n}_s)-\psi(X^{x,n}_{s-})|\leq c|\Delta X^{x,n}_s|=c|\Delta N_s|$, we have
\[
E|I^{(2)}_t|^p=E\Big|\sum_{s\leq t,\,|\Delta N_s|>1}(\psi(X^{x,n}_s)-\psi(X^{x,n}_{s-}))\Big|^p
\leq cE\Big|\sum_{s\leq t,\,|\Delta N_s|>1}|\Delta N_s|\Big|^p.
\]
Hence
\begin{align*}
E|I^{(2)}_t|^p&\le cE\Big|\int_{\{|x|>1\}}|x|\,(\mu(t,dx)-t\nu(dx))
+t\int_{\{|x|>1\}}|x|\,\nu(dx)\Big|^p\nonumber\\
&\le c2^{p-1}E\Big|\int_{\{|x|>1\}}|x|\,(\mu(t,dx)-t\nu(dx))\Big|^p
+c2^{p-1}\Big|t\int_{\{|x|>1\}}|x|\,\nu(dx)\Big|^p.
\end{align*}
Applying the Burkholder--Davis--Gundy inequality we get
\begin{equation}
\label{eq3.004}
E|I^{(2)}_t|^p\le cE\Big(\sum_{s\leq t,\,|\Delta N_s|>1}|\Delta N_s|^p +t^p\Big)
=c\Big(t\int_{\{|x|>1\}}|x|^p\,\nu(dx))+t^p\Big).
\end{equation}
Similarly,
\begin{align*}
E|I^{(3)}_t|^p&=E\Big|\sum_{s\leq t,\,|\Delta N_s|\leq1}(\psi(X^{x,n}_s)-\psi(X^{x,n}_{s-})-\nabla\psi(X^{x,n}_{s-})\Delta X^{x,n}_s)\Big|^p\nonumber\\
&\le cE\Big|\sum_{s\leq t,\,|\Delta N_s|\leq1}|\Delta N_s|^2\Big|^p\nonumber\\
&\le cE\Big|\int_{\{|x|\leq1\}}|x|^2\,(\mu(t,dx)-t\nu(dx))
+t\int_{\{|x|\leq1\}}|x|^2\,\nu(dx)\Big|^p\nonumber\\
&\leq c2^{p-1}E\Big|\int_{\{|x|\leq1\}}|x|^2\,(\mu(t,dx)-t\nu(dx))\Big|^p
+c2^{p-1}\Big|t\int_{\{|x|\leq1\}}|x|^2\,\nu(dx)\Big|^p.
\end{align*}
Hence
\begin{align}
\label{eq3.005}
E|I^{(3)}_t|^p&\le cE\Big(\sum_{s\leq t,\,|\Delta N_s|\leq1}|\Delta N_s|^4\Big)^{p/2}
+ct^p \nonumber\\
&\le cE\Big(\sum_{s\leq t,\,|\Delta N_s|\leq1}|\Delta N_s|^{2p} +t^p\Big)
=c\Big(t\int_{\{|x|\leq1\}}|x|^{2p}\,\nu(dx)+t^p\Big).
\end{align}
Putting together (\ref{eq3.003})--(\ref{eq3.005}) and observing  that
$c(t)=\max(t^p,t^{p/2},t)$ completes the proof of the desired estimate of $E|K^{x,n}|^p_t$.
Using \cite[Corollary 3.6]{LS1} and arguments from the proof of Theorem \ref{th2.4}  shows  that
$|K^{x,n}|_t\rightarrow|K^{x}|_t$ in probability. Therefore applying Fatou's lemma we get
\[
E|K^x|_t^p\leq\liminf_{n\rightarrow\infty}E|K^{n,x}|_t^p\leq C(1+c(t)),
\]
which completes the proof.
\end{proof}

\begin{proposition}
\label{prop3.5}
Under the assumptions of Theorem \ref{th3.3} there is $C>0$ independent of $t$ such that
\[
\sup_{n\ge1}E\sup_{s\leq t}|X^{x,n}_s-x|^p\leq C(1+c(t)),\quad t\ge0,
\]
with  $c(t)$ defined in Theorem \ref{th3.3}.
\end{proposition}
\begin{proof}
Clearly, for every $t\ge0$,
\[
E\Big|\int^{t}_0b_i(X^{x,n}_s)\,ds\Big|^{p}\le ct^{p},
\]
and by the Davis inequality,
\[
E\Big|\int^t_0\sigma_{ij}(X^{x,n})\,dW^j_s\Big|^p\le cE\Big[\int^{\cdot}_0|\sigma_{ij}(X^{x,n}_s)\,dW^j_s\Big]^{p/2}_t
\le ct^{p/2}.
\]
Furthermore, as in the proof of (\ref{eq3.003}) and (\ref{eq3.004}) we get
\[
E|N^{(2)}_t|^p\le cE[N^{(2)}_t]^{p/2}\le Ct,\qquad E|N^{(3)}_t|^p\le c(t+t^p), \quad t\ge0.
\]
which completes the proof.
\end{proof}

\begin{corollary}
\label{cor3.6}
Under the assumptions of Theorem  \ref{th3.3} for every compact  $K\subset\BR^d$
there is a constant $C>0$ independent of $t,n$  such that
\[
\sup_{x\in K,\,n\ge1}\max\big(E|K^{x}|^p_t,E\sup_{s\leq t}|X^{x,n}_s|^p,E|K^{x,n}|^p_t\big)
\leq C(1+c(t)),\quad t\ge0,
\]
with  $c(t)$ defined in Theorem \ref{th3.3}.
\end{corollary}
\begin{proof}
Follows from  Theorem \ref{th3.3}, Proposition \ref{prop3.5} and the  fact that the constant $c_0$ from the proof of Theorem \ref{th3.3}, which
depends on the initial point (the other  estimates do not depend on $x$)  has the form $c_0=\mbox{dist}(x,\bar D)+\sup_{x\in\bar D}\psi(x)$.
\end{proof}

\section{Neumann problem}
\label{sec4}

In this section, we first show  (\ref{eq1.7}) for probabilistic solutions and {\em a fortiori} the existence of the probabilistic solution $u$ of (\ref{eq1.1}). Then we show that probabilistic solutions coincide with viscosity solutions defined in \cite{BGJ}, so  (\ref{eq1.7}) is nothing but the statement on convergence of viscosity solutions. We will assume that the operator $L$ satisfies conditions  (A1)--(A3). As for the data $f$ and $g$, we will assume that
\begin{enumerate}
\item[\rm(A4)] $g:\BR^d\rightarrow\BR$ is continuous  and bounded, $f:\BR^d\rightarrow\BR$ is continuous and $|f(x)|\le K(1+|x|^p)$, $x\in\BR^d$, for some $K>0$ (and $p$ of condition (A3)).
\end{enumerate}


\subsection{Probabilistic solutions}

Probabilistic solution of (\ref{eq1.1}) at point $x\in\BR^d$ is defined as a Feynman--Kac type functional of the solution $(X^x,K^x)$ of (\ref{eq3.1}).

\begin{definition}
$u:\BR^d\rightarrow\BR$ is called a probabilistic solution of (\ref{eq1.1}) if
\begin{align}
\label{eq4.22}
u(x)&=E\int_{[0,\infty)}e^{-\lambda t}(f(X^{x}_t)\,dt +g(X^{x}_t)\,d\|K^{x}\|_t)\nonumber\\
&\quad+E\sum_{0\leq t<\infty}e^{-\lambda t}
\int_0^{|\Delta K^x_t|}\bar g(X^x_t-s{\mathbf n}(X^x_t))\,ds,\quad x\in\BR^d,
\end{align}
where $(X^x,K^x)$ is the  solution of (\ref{eq3.1}) and $\bar g$ is defined by \mbox{\rm(\ref{eq1.10})}.  
\end{definition}

\begin{remark}
\label{rem4.2}
(i) By (\ref{eq2.15}), $\Delta K^x_0=\Pi(x)-x$ and $\Delta K^x_t=\Delta K^{\Pi(x)}_t$, $t>0$. It follows in particular that
\[
\int_0^{|\Delta K^x_0|}\bar g(X^x_0-s{\mathbf n}(X^x_0))\,ds
=\int^{|x-\Pi(x)|}_0g(\Pi(x)-s{\mathbf n}(\Pi(x)))\,ds-g(\Pi(x))|x-\Pi(x)|.
\]
Therefore (\ref{eq2.15}) implies that an equivalent definition of a probabilistic solution of (\ref{eq1.1}) is the following: for every $x\in\BR^d$,
\begin{align}
\label{eq4.25}
u(x)&=E\int^{\infty}_0e^{-\lambda t} (f(X^{\Pi(x)}_t)\,dt
+g(X^{\Pi(x)}_t)\,d|K^{\Pi(x)}|_t) \nonumber\\
&\quad+\int_0^{|x-\Pi(x)|}g(\Pi(x)-s{\mathbf n}(\Pi(x)))\,ds \nonumber\\
&\quad+E\sum_{0<t<\infty}e^{-\lambda t} \int_0^{|\Delta K^{\Pi(x)}_t|}
\bar g(X^{\Pi(x)}_t-s{\mathbf n}(X^{\Pi(x)}_t))\,ds.
\end{align}
(ii) From (\ref{eq4.22}), (\ref{eq4.25}) and the fact that $\Delta K^{\Pi(x)}_0=0$ it follows that
\begin{equation}
\label{eq4.16}
u(x)=u(\Pi(x))+\int_0^{|x-\Pi(x)|}g(\Pi(x)-s{\mathbf n}(\Pi(x)))\,ds,\quad x\in\BR^d.
\end{equation}
(iii) Suppose that $u\in C(\BR^d)$. Then from (\ref{eq4.16})
and the fact that $\nu$ is a L\'evy measure it follows that for bounded $D$ and  $g$ we have 
\begin{equation}
\label{eq4.44}
	\int_{|y|\ge\delta}|u(x+y)|\,\nu(dy)<\infty
\end{equation} 	
for any $x\in\bar D$ and $\delta>0$.
\end{remark}

We will approximate $u$ by probabilistic solutions $u_n$ of the problems
\begin{equation}
\label{eq4.1}
-L^nu_n+\lambda u_n=f+ng
 \cdot\mbox{dist}(\,\cdot,\bar D)\quad\mbox{in }\BR^d,
\end{equation}
where $L^n$ is the operator associated with the solution of (\ref{eq3.2}), i.e.
\begin{equation}
\label{eq4.2}
L^n=L-n\sum^d_{i=1}(x_i-\Pi^i(x))\frac{\partial}{\partial x_i}.
\end{equation}

\begin{definition}
$u_n:\BR^d\rightarrow\BR$ is called a probabilistic solution of (\ref{eq4.1})
if
\begin{align}
\label{eq4.12}
u_n(x)&=E\int^{\infty}_0e^{-\lambda t}
\Big(f(X^{x,n}_t)+ng(X^{x,n}_t)
\cdot\mbox{dist}(X^{x,n}_t,\bar D)\,dt\Big) \nonumber\\
&=E\int^{\infty}_0e^{-\lambda t}
\Big(f(X^{x,n}_t)\,dt+g(X^{x,n}_t)\,d|K^{x,n}|_t\Big),
\end{align}
where $X^{x,n}$ is the  solution of (\ref{eq3.2}).
\end{definition}

\begin{remark}
\label{rem4.4}
From Proposition \ref{prop3.5} and (\ref{eq4.12}) it follows that if (A1)--(A4) are satisfied, then for any $x\in\BR^d$ and $\delta>0$,
\[
\int_{|y|\ge1}|u_n(x+y)|\,\nu(dy)<\infty.
\]	
\end{remark}

In the sequel,  $B(x,r)$ stands for the open ball of radius $r>0$ with center at $x\in\BR^d$,  and $\bar B(x,r)$ stands for the closure of $B(x,r)$.

\begin{theorem}
\label{th4.4}
Assume that \mbox{\rm (A1)--(A4)} are satisfied. Let $u$ be the probabilistic solution of \mbox{\rm(\ref{eq1.1})} and $u_n$ be the probabilistic solution of \mbox{\rm(\ref{eq4.1})}. Then $u, u_n\in C(\BR^d)$  and for every compact subset  $K\subset\BR^d$,
\begin{equation}
\label{eq4.26}
\lim_{n\rightarrow\infty}\sup_{x\in K}|u_n(x)- u(x)|=0.
\end{equation}
\end{theorem}
\begin{proof}
We first note that  by Corollary \ref{cor3.6}
there is a constant $C>0$ independent of $T\ge1$  such that
\begin{equation}
\label{eq4.06}
\sup_{y\in B(x,1),\,n\ge1}\max\big(E|K^{y}|^p_T,E\sup_{s\leq T} |X^{y,n}_s|^p,E|K^{y,n}|^p_T\big)\leq C(1+T^p).
\end{equation}
To show that $u\in C(\BR^d)$ we fix $x\in\BR^d$ and assume that $B(x,1)\ni y\to x$. Let $(X^x,K^x)$ be the unique solution of (\ref{eq3.1}), i.e.  the  unique solution of the  Skorokhod problem associated with  $Y^{x}$ defined by
\begin{equation}\label{eq4.07}
Y^{x,i}_t=x_i+\sum^d_{j=1}\int^t_0\sigma_{ij}(X^{x}_s)\,dW^j_s
+\int^t_0b_i(X^{x}_s)\,ds  +N^i_t,\quad t\ge0,
\end{equation}
Let  $(X^y,K^y)$ denote the unique solution of (\ref{eq3.1})  with $x$ replaced by $y$. By \cite[Theorem 4, Corollary 11]{Sl}, $Y^{y}\rightarrow Y^x$ in probability in the space $\BD(\BR_+;\BR^d)$ equipped with the $J_1$-topology. Moreover, by (\ref{eq4.06}), for every $T>0$ the family $\{|K^y|_T:y\in B(x,1)\}$ is uniformly integrable. Also, for every $T>0$, $P(\Delta Y^{x}_T=0)=P(\Delta N_T=0)=1$ (see, e.g., \cite[Lemma 2.3.2]{A}). Therefore, by Theorem \ref{th2.10}, for every $T>0$ we have
\begin{equation}
\label{eq4.01}
E\int_0^Te^{-\lambda t}f(X^y_t)\,dt\rightarrow E\int_0^Te^{-\lambda t}f(X^x_t)\,dt
\end{equation}
and
\begin{equation}\label{eq4.02}
EI_T(g_\lambda,X^{y},K^{y})\rightarrow EI_T(g_\lambda,X^{x},K^{x}),
\end{equation}
where $g_\lambda(t,x)=e^{-\lambda t}g(x)$, $t\geq0$, $x\in\BR^d$. Since $X^y\in\bar D$ for $y\in B(x,1)$, by taking  large $T>0$ the integrals  $E\int^{\infty}_Te^{-\lambda t}|f(X^y_t)|\,dt$ can  be made arbitrarily small uniformly in $y$. By this and  (\ref{eq4.01}),
\begin{equation}
\label{eq4.27}
E\int_0^\infty e^{-\lambda t}f(X^y_t)\,dt\rightarrow
E\int_0^\infty e^{-\lambda t}f(X^x_t)\,dt.
\end{equation}
Furthermore, by the integration by parts formula,
\[
E\int^t_0e^{-\lambda s}\,d|K^{y}|_s=e^{-\lambda t}E|K^{y}|_t
+\lambda E\int^t_0e^{-\lambda s}|K^{y}|_s\,ds.
\]
Letting $t\rightarrow\infty$ and using (\ref{eq4.06}) shows that
\begin{equation}
\label{eq4.24}
E\int^{\infty}_0e^{-\lambda t}\,d|K^{y}|_t
=\lambda E\int^{\infty}_0e^{-\lambda t}|K^{y}|_t\,dt<\infty.
\end{equation}
We have
\[
E\int^{\infty}_Te^{-\lambda t}g(X^{y}_t)\,d|K^{y}|_t
\le\|g\|_\infty\Big(-e^{-\lambda T}E|K^{y}|_T
+\lambda E\int^{\infty}_Te^{-\lambda t}|K^{y}|_t\,dt\Big).
\]
By (\ref{eq4.06}) and (\ref{eq4.24}) the right-hand side of the above inequality converges to zero as $T\rightarrow\infty$ uniformly in $y\in B(x,1)$.
Moreover,
\begin{align*}
E\sum_{t>T}e^{-\lambda t}\int_0^{|\Delta K^y_t|}|\bar g(X^y_t-s{\mathbf n}(X^y_t))|\,ds
&\leq 2\|g\|_\infty E\sum_{t>T}e^{-\lambda t}|\Delta K^y_t|\\
&\leq 2\|g\|_\infty E\int_T^\infty e^{-\lambda t}d|K^y|_t
\end{align*}
and the right-hand side of the above inequality also converges to zero as $T\rightarrow\infty$ uniformly in $y\in B(x,1)$.
Therefore from  (\ref{eq4.02}) and (\ref{eq4.27}) it follows that $u(y)\rightarrow u(x)$.
The proof of continuity of $u_n$ is  similar. We fix $x\in\BR^d$ and assume that $B(x,1)\ni y\rightarrow x$.
Let $X^{x,n}$ be the unique  solution of (\ref{eq3.2}) and $Y^{x,n}$ be defined by (\ref{eq3.13}) that is
\[
Y^{x,n,i}_t=x_i+\sum^d_{j=1}\int^t_0\sigma_{ij}(X^{x,n}_s)\,dW^j_s
+\int^t_0b_i(X^{x,n}_s)\,ds  +N^i_t=X^{x,n,i}_t-K^{x,n,i}_t.
\]
We denote by  $X^{y,n}$ the unique  solution of (\ref{eq3.2}) with $x$ replaced by $y$.
By \cite[Theorem 1(ii)]{S2}, $(X^{y,n},K^{y,n})\rightarrow (X^{x,n},K^{x,n})$ in probability in the space $\BD(\BR_+;\BR^{2d})$ equipped with  $J_1$-topology. By (\ref{eq4.06}), for every $T>0$ the families $\{\sup_{t\leq T}|X^{y,n}_t|^p:y\in B(x,1)\}$ and $\{|K^{y,n}|_T:y\in B(x,1)\}$ are uniformly integrable. By the arguments from the proof of Theorem \ref{th2.5}  and Theorem \ref{th2.10}, for all $T>0$ and $n\ge1$ we have
\[
E\int_0^Te^{-\lambda t}f(X^{y,n}_t)\,dt\rightarrow E\int_0^Te^{-\lambda t}f(X^{x,n}_t)\,dt
\]
and
\[
E\int^{T}_0e^{-\lambda t}g(X^{y,n}_t)\,d|K^{y,n}|_t
\rightarrow E\int^{T}_0e^{-\lambda t}g(X^{x,n}_t)\,d|K^{x,n}|_t.
\]
We also have
\begin{equation}
\label{eq4.11}
E\int_T^{\infty}e^{-\lambda t}|f(X^{y,n}_t)|\,dt
\le C E\int_T^{\infty}e^{-\lambda t}(1+|X^{y,n}_t|^p)\,dt.
\end{equation}
By this and (\ref{eq4.06}),
\begin{equation}
\label{eq4.32}
\lim_{T\to\infty}\sup_{y\in B(x,1)}\int_T^{\infty}e^{-\lambda t} |f(X^{y,n}_t)|\,dt=0,
\end{equation}
Furthermore, analysis similar to that in the proof of (\ref{eq4.32}) shows that
\begin{equation}
\label{eq4.34}
\lim_{T\to\infty}\sup_{y\in B(x,1)}\int_T^\infty |g(X^{y,n}_t)|\,d|K^{y,n}|_t=0,
\end{equation}
which completes the proof of the convergence $u_n(y)\rightarrow u_n(x)$.
We now turn to the proof of (\ref{eq4.26}). Since we know that the limit function is continuous, it suffices to show that for every $x$  and every sequence $\{x_n\}$, if $x_n\to x$ then
\begin{equation}
\label{eq4.03} \lim_{n\rightarrow\infty}u_n(x_n)=u(x).
\end{equation}
To check (\ref{eq4.03}), we first note that by \cite[Theorem 3.5, Corollary 3.6]{LS1}, $Y^{x_n,n}\rightarrow Y^x$ in probability  the $J_1$-topology. Moreover, by  Theorem \ref{th2.9} and (\ref{eq4.06}), for every  $T>0$,
\begin{equation}
\label{eq4.04}
\lim_{n\rightarrow\infty}E\int_0^Te^{-\lambda t}f(X^{x_n,n}_t)\,dt
=E\int_0^Te^{-\lambda t}f(X^x_t)\,dt
\end{equation}
and
\begin{equation}
\label{eq4.05}
\lim_{n\rightarrow\infty}E\int_0^Te^{-\lambda t}g(X^{x_n,n}_t)\,d|K^{x_n,n}|_t
=EI_T(g_\lambda,X^{x},K^{x}).
\end{equation}
By (\ref{eq4.06}) and  the arguments used to prove (\ref{eq4.32}) and (\ref{eq4.34}) we also have
\[
\lim_{T\to\infty}\sup_{n\ge1}E\int_T^{\infty}e^{-\lambda t}(|f(X^{x_n,n}_t)|\,ds+|g(X^{x_n,n}_t)|\,d|K^{x_n,n}|_t)=0,
\]
which when combined with (\ref{eq4.04}) and (\ref{eq4.05}) yields
(\ref{eq4.03}).
\end{proof}

\begin{remark}
By (\ref{eq2.1}),
the operator $L^n$ defined by (\ref{eq4.2}) can be written in the form
\[
L^n=L+n\mbox{dist}(x,\bar D)\frac{\partial}{\partial\bar{\mathbf n}}.
\]
Consequently, (\ref{eq4.1}) can be written in the form
\begin{equation}
\label{eq4.3}
F(x,u_n,\nabla u_n,\nabla^2u_n,\II[u_n](x))-n\mbox{dist}(x,\bar D)
\Big(\frac{\partial u_n}{\partial\bar{\mathbf n}}
+g\Big)(x)=0,\quad x\in\BR^d,
\end{equation}
where
\begin{equation}
\label{eq4.13}
F(x,r,p,M,l)=-\frac12\mbox{Tr}(a(x)M)-b(x)\cdot p+\lambda r-l-f(x).
\end{equation}
Note that (\ref{eq4.3}) reduces to the penalized equation considered in \cite{BGJ} with the exception that in that paper $\mbox{dist}(x,\bar D)$ is replaced by
$\min(\mbox{dist}(x,\bar D),1)$.
\end{remark}

By way of illustration, below we consider two very simple examples of  (\ref{eq1.1}).
For these problems one can  show  by elementary computations  why in (\ref{eq4.22}) the second term on the right-hand side appears.

\begin{example}
Suppose that $a=0$, $b=0$, $\nu=0$ and $f=0$. Then $Lu_n=0$ and (\ref{eq4.3}) (or (\ref{eq4.1})) reduces to
\begin{equation}
\label{eq4.23}
\lambda u_n=n\mbox{dist}(x,\bar D)
\Big(\frac{\partial u_n}{\partial\bar{\mathbf n}}
+g\Big)(x),\quad x\in\BR^d.
\end{equation}
(a) As in Example \ref{ex2.5}, let $D=(0,1)$ and $x=-z$ for $z>0$. Then the process $Y^x$ of (\ref{eq3.2}) is nothing but  the deterministic process $y_t=-z$, so $(X^x,K^x)$ is the deterministic process $(x,k)$ of Example \ref{eq2.5} and  $(X^{x,n},K^{n,x})$ of (\ref{eq3.2}) is the process $(x^n,k^n)$. By this and (\ref{eq4.25}),
\[
u(x)= \int_0^{|x-\Pi(x)|}g(\Pi(x)-{\mathbf n}(\Pi(x))s)\,ds=\int_0^zg(-s)\,ds
\]
and by (\ref{eq4.12}),
\begin{align*}
u_n(x)=\int^{\infty}_0e^{-\lambda s}g(x^n_s)\,d|k^n|_s&=\int_0^\infty e^{-\lambda s} g(-ze^{-ns})ne^{-ns}ds\\
&=-\int_{z}^0s^{\lambda/n}g(-s)\,dr\rightarrow\int_{0}^zg(-s)\,ds.
\end{align*}
(b) Let $d=2$, $D=B(0,1)$. Set $x=(x_1,x_2)\notin \bar D$,
$z=\mbox{\rm dist}(x,\bar D)=\|x\|-1>0$. In this case
\[
\Pi(x_1,x_2)=\Big(\frac{x_1}{\|x\|},\frac{x_2}{\|x\|}\Big),
\qquad\bar{\mathbf n}(x_1,x_2)
={\mathbf n}(\Pi({x_1},{x_2}))=\Big(-\frac{x_1}{\|x\|},-\frac{x_2}{\|x\|}\Big)
\]
and $(X^x,K^x)$ is the deterministic process $(x,k)$ defined as   $x_t=\Pi(x_1,x_2)$, $k_t=\Pi(x_1,x_2)-(x_1,x_2)$ for $t\geq0$ with the convention that $k_{0-}=0$.
Let $g\in C(\BR^2)$. By Remark \ref{rem4.2}(i),
\[
u(x_1,x_2)= \int_0^{z}g(\Pi(x_1,x_2)-s{\mathbf n} (\Pi(x_1,x_2)))\,ds=\int_0^zg(\Pi(x_1,x_2)(1+s))\,ds.
\]
Moreover, the process $(X^{x,n},K^{n,x})$ of (\ref{eq3.2}) is  the deterministic process $(x^n,k^n)$ of the form
 $x^n_t=\Pi(x_1,x_2)+((x_1,x_2)-\Pi(x_1,x_2))e^{-nt}$, $k^n_t=-n\int^t_0(x^n_s-\Pi(x_1,x_2))\,ds$, $t\ge0$.
By (\ref{eq4.12}),
\begin{align*}
u_n(x_1,x_2)&=\int^{\infty}_0e^{-\lambda s}g(x^n_s)\,d|k^n|_s\\
&=\int_0^\infty e^{-\lambda s}g(\Pi(x_1,x_2)+((x_1,x_2)-\Pi(x_1,x_2))e^{-ns})nze^{-ns}ds\\
&=-\int_{z}^0s^{\lambda/n}g(\Pi(x_1,x_2)-s{\mathbf n}(\Pi(x_1,x_2)))\,ds\rightarrow\int_{0}^zg(\Pi(x_1,x_2)(1+s))\,ds.
\end{align*}
\end{example}

\subsection{Viscosity solutions}

Let $x\in\BR^d$. For $\varphi\in C^2(\BR^d)$ and  $u\in C(\BR^d)$ satisfying (\ref{eq4.44}) we set
\[
\II_{\delta}[\varphi](x)=\int_{|y|<\delta}(\varphi(x+y)-\varphi(x)-y\cdot\nabla \varphi(x)\fch_{B(0,1)}(y))\,\nu(dy),
\]
\[
\II^{\delta}[u,\varphi](x)=\int_{|y|\ge\delta}(u(x+y)-u(x)-y\cdot\nabla \varphi(x)\fch_{B(0,1)}(y))\,\nu(dy)
\]
and
\[
F^{\delta}[u,\varphi](x)=F(x,u,\nabla\varphi(x),\nabla^2\varphi(x),
\II_{\delta}[\varphi](x)+\II^{\delta}[u,\varphi](x)),
\]
where  $F$ is defined by (\ref{eq4.13}). 
Following \cite[Definition 1]{BI} we adopt the following definition of a viscosity solution of (\ref{eq4.1}).  

\begin{definition}
(i) A function $u_n\in C(\BR^d)$ satisfying (\ref{eq4.44}) is a viscosity subsolution of (\ref{eq4.3}) if, for any test function $\varphi\in C^2(\BR^d)$, if $x\in\BR^d$ is a  maximum point of $u_n-\varphi$ in $B(x,\delta)$, then
\begin{equation}
\label{eq4.4}
F^{\delta}[u_n,\varphi](x)
-n\mbox{dist}(x,\bar D)
\Big(\frac{\partial \varphi}{\partial\bar{\mathbf n}}+g\Big)(x)\le 0
\end{equation}
for every $\delta>0$. A function $u_n\in C(\BR^d)$ is a viscosity supersolution of (\ref{eq4.3}) if, for any test function $\varphi\in C^2(\BR^d)$, if $x\in\BR^d$ is a minimum point of  $u_n-\varphi$ in $B(x,\delta)$, then
\[
F^{\delta}[u_n,\varphi](x)
-n\mbox{dist}(x,\bar D)
\Big(\frac{\partial \varphi}{\partial\bar{\mathbf n}}+g\Big)(x)\ge0.
\]
(ii) $u$ is called a viscosity solution of (\ref{eq4.1}) if it is both a viscosity sub and supersolution of (\ref{eq4.1}).
\end{definition}

The following lemma is an adaptation of part of \cite[Lemma 3.3]{BBP} to our situation (see also \cite[Proposition 1]{BI}).

\begin{lemma}
\label{lem4.8}
Let $u_n$ satisfy \mbox{\rm (\ref{eq4.44})}. Suppose that $u_n\in C(\BR^d)$   has the property that for any $\varphi\in C^2(\BR^d)$, whenever $x$ is a global maximum point of $u_n-\varphi$, then 
\begin{equation}
\label{eq4.39}
F(x,u_n,\nabla\varphi(x),\nabla^2\varphi(x),
\II[\varphi](x))-n\mbox{\rm dist}(x,\bar D)
\Big(\frac{\partial \varphi}{\partial\bar{\mathbf n}}+g\Big)(x)\le 0.
\end{equation}
Then $u_n$ satisfies \mbox{\rm(\ref{eq4.4})}. Similarly, suppose that $u_n\in C(\BR^d)$ has the property that for any  $\varphi\in C^2(\BR^d)$, whenever $x$ is a global minimum point of $u_n-\varphi$, then 
\mbox{\rm(\ref{eq4.39})} is satisfied with ``\,$\le$'' replaced by ``\,$\ge$''. Then $u_n$ satisfies \mbox{\rm(\ref{eq4.4})} with  ``\,$\le$'' replaced by ``\,$\ge$''.
\end{lemma}  

Note that (\ref{eq4.4}) can be equivalently stated as
\begin{equation}
\label{eq4.5}
-L^n\varphi(x)+\lambda u_n(x)\le f(x)+n\mbox{dist}(x,\bar D)g(x).
\end{equation}

We are going to show that the probabilistic solution of (\ref{eq4.1}) is a viscosity solution. The proof is rather standard but we could not find a proper reference
to the case we need.   We provide a proof for completeness.

\begin{proposition}
\label{prop4.3}
Let $u_n$ be a probabilistic solution of \mbox{\rm(\ref{eq4.1})}. Then $u_n$ is a  viscosity solution of \mbox{\rm(\ref{eq4.1})}.
\end{proposition}
\begin{proof}
From Remark \ref{rem4.4} we know that  (\ref{eq4.44}) is satisfied for $x\in\BR^d$ and $\delta>0$.  We shall use Lemma \ref{lem4.8} to show that $u_n$ is a viscosity subsolution of (\ref{eq4.1}).
Suppose that $x$ is a global maximum of $u_n-\varphi$. Without loss of generality we can assume that $u_n(x)=\varphi(x)$. Set $h_n(y)=f(y)+n\mbox{dist}(y,\bar D)g(y)=f(y)+n|\Pi(y)-y|g(y)$, $y\in\BR^d$, and
\[
M^{x,n}_t=E\Big(\int^{\infty}_0
e^{-\lambda s}h_n(X^{x,n}_s)\,ds\,\big|\,\FF_t\Big)-u_n(X^{x,n}_0), \quad t\ge0.
\]
Then $M^{x,n}$ is a uniformly integrable martingale. We have
\[
M^{x,n}_t=\int^t_0e^{-\lambda s}h_n(X^{x,n}_s)\,ds-u_n(X^{x,n}_0)
+E\Big(\int^{\infty}_te^{-\lambda s}h_n(X^{x,n}_s)\,ds|\FF_t\Big).
\]
From the Markov property (see \ref{eq3.4})) it follows that
\begin{align*}
E\Big(\int^{\infty}_te^{-\lambda s}h_n(X^{x,n}_s)\,ds\,\big|\,\FF_t\Big)
&=E\Big(\int^{\infty}_0e^{-\lambda(s+t)}h_n(X^{x,n}_{s+t})\,ds\,\big|\,\FF_t\Big)\\
&=E\int^{\infty}_0e^{-\lambda(s+t)}h_n(X^{y,n}_s)\,ds\,\big|_{y=X^{x,n}_t}
=e^{-\lambda t}u_n(X^{x,n}_t).
\end{align*}
As a result we have
\begin{equation}
\label{eq4.6}
u_n(X^{x,n}_0)=e^{-\lambda t}u_n(X^{x,n}_t)+\int^t_0e^{-\lambda s}h_n(X^{x,n}_s)\,ds-M^{x,n}_t,
\quad t\ge0.
\end{equation}
Let
\[
\bar Y^{x,n}_t=e^{-\lambda t}\varphi(X^{x,n}_t),\quad t\ge0.
\]
Applying  It\^o's formula we see that there is a local martingale
$\bar M^{x,n}$ such that for every $t>0$,
\begin{equation}
\label{eq4.9}
\bar Y^{x,n}_0=e^{-\lambda t}\bar Y^{x,n}_t +\int^t_0e^{-\lambda s}(-L^n\varphi(X^{x,n}_s)
+\lambda\varphi(X^{x,n}_s))\,ds+\int^t_0d\bar M^{x,n}_s.
\end{equation}
Striving for a contradiction, suppose that (\ref{eq4.5}) is not satisfied. Then there is
$r>0$ such that $-L^n\varphi(y)+\lambda u_n(y)>h_n(y)$ for $y\in \bar B(x,r)$. Set $\tau=\inf\{t\ge0:|X^{x,n}_t-x|>r\}$. Note that $P(\tau>0)=1$ since $X^{x,n}$ has right-continuous sample paths. By (\ref{eq4.6}) and (\ref{eq4.9}),
\begin{align}
\label{eq4.7}
u_n(x)=E\Big(e^{-\lambda\tau}u_n(X^{x,n}_{\tau-})
+\int^{\tau-}_0e^{-\lambda s} h_n(X^{x,n}_s)\,ds\Big)
\end{align}
and
\begin{equation}
\label{eq4.8}
E\bar Y^{x,n}_0=E\Big(e^{-\lambda \tau}\varphi(X^{x,n}_{\tau-})
+\int^{\tau-}_0e^{-\lambda s} 
(-L^n\varphi(X^{x,n}_s)+\lambda\varphi(X^{x,n}_s))\,ds 
\end{equation}
(We first get (\ref{eq4.8}) with $\tau$ replaced by $\sigma^n_k\wedge\tau$ for a 
localizing sequence $(\sigma^n_k)_{k\ge1}$ for $\bar M^{x,n}$, and then letting $k\rightarrow\infty$ we get (\ref{eq4.8})).
For $y\in\bar B(x,r)$ we have $(u_n-\varphi)(y)\le (u_n-\varphi)(x)=0$ and
$h_n(y)<-L^n\varphi(y)+\lambda u_n(y)\le -L^n\varphi(y)+\lambda\varphi(y)$. Furthermore $X^{x,n}_t\in\bar B(x,r)$, $t\in[0,\tau)$, $P$-a.s. Therefore
from (\ref{eq4.7}) and (\ref{eq4.8})  it follows that $u_n(x)<E\bar Y^{x,n}_0=\varphi(x)$,
which is a contradiction. In the same way we show that $u_n$ is a viscosity supersolution of (\ref{eq4.1}).
\end{proof}

Following \cite{BGJ} we adopt the following definition of a continuous viscosity solution of (\ref{eq1.1}).

\begin{definition}
(i) $u\in C(\BR^d )$  (\ref{eq4.44}) for $x\in\bar D$ is a viscosity subsolution of (\ref{eq1.1})
if, for any $\varphi\in C^2(\BR^d)$, if $x\in\BR^d$ is a  maximum point of $u-\varphi$ in $B(x,\delta)$, then
\begin{equation}
\label{eq4.17}
\begin{cases} 
F^{\delta}[u,\varphi](x)\le0 & \mbox{if }x\in D,\\
\min\big(F^{\delta}[u,\varphi](x),
-\frac{\partial\varphi}{\partial{\mathbf n}}(x)-g(x)\big)\le0 & \mbox{if } x\in\partial D,\\
-\frac{\partial\varphi}{\partial{\bar{\mathbf n}}}(x)\le g(x) & \mbox{if } x\in\BR^d\setminus \bar D
\end{cases}
\end{equation}
for every $\delta>0$. A function $u\in C(\BR^d)$ satisfying (\ref{eq4.44}) for $x\in\bar D$ is a viscosity supersolution
of (\ref{eq1.1}) if, for any  $\varphi\in C^2(\BR^d)$, if $x\in\BR^d$ is a  minimum point of $u-\varphi$ in $B(x,\delta)$, 
then
\begin{equation}
\label{eq4.18}
\begin{cases}
F^{\delta}[u,\varphi](x)\ge0 & \mbox{if }x\in D,\\
\max\big(F^{\delta}[u,\varphi](x),
-\frac{\partial\varphi}{\partial{\mathbf n}}(x)-g(x)\big)\ge0 & \mbox{if } x\in\partial D,\\
-\frac{\partial\varphi}{\partial{\bar{\mathbf n}}}(x)\ge g(x) &
\mbox{if }x\in\BR^d\setminus \bar D
\end{cases}
\end{equation}
for every $\delta>0$.\\
(ii) $u$ is called a viscosity solution of (\ref{eq1.1}) if it is both a viscosity sub and supersolution.
\end{definition}

\begin{remark}
\label{rem4.11}
As in case of equation (\ref{eq4.3}) considered in Lemma \ref{lem4.8},  to check that $u$ is a viscosity subsolution (resp. supersolution) of (\ref{eq1.1}) it suffices to check that (\ref{eq4.17}) (resp. (\ref{eq4.18})) holds with $\II^{\delta}[u,\varphi](x)$ appearing in the definition of $F^{\delta}[u,\varphi]$ replaced by $\II^{\delta}[\varphi,\varphi](x)$ whenever $x$ is a global maximum (resp. minimum) point of $u-\varphi$ 
(see \cite[Remark 3.2]{BGJ}).
\end{remark}

\begin{proposition}
\label{prop4.10}
Let $u$ be a probabilistic solution of \mbox{\rm(\ref{eq1.1})}. Then $u$ is a viscosity solution of \mbox{\rm(\ref{eq1.1})}.
\end{proposition}
\begin{proof}
First note that by Remark \ref{rem4.2}(iii), $u$ satisfies (\ref{eq4.44}) for every $x\in\bar D$. To shorten notation, write
\[
S_t(g,X^x,K^x)=\sum_{0\le s\le t}e^{-\lambda s}\int^{|\Delta K^x_s|}_0
\bar g(X^x_s-r{\mathbf n}(X^x_s))\,dr, \quad t\ge0.
\]
Define  $S_{\infty}(g,X^x,K^x)$ as $S_t$ but with $\sum_{0\le s\le t}$ replaced by $\sum_{0\le s<\infty}$, and then set
\begin{equation}
\label{eq4.47}
M^x_t=E\Big(\int^{\infty}_0e^{-\lambda s}f(X^x_s)\,ds
+\int^{\infty}_0e^{-\lambda s} g(X^x_s)\,d|K^x|_s+S_{\infty}(g,X^x,K^x)\,\big|\,\FF_t\Big)-u(X^x_0).
\end{equation}
Then $M^x$ is a uniformly integrable martingale. By using the Markov property (see the remarks following Proposition  \ref{prop3.3}) we check that
\begin{equation}
\label{eq4.43}
E\Big(\int^{\infty}_te^{-\lambda s}f(X^x_s)\,ds\,\big|\FF_t\Big)
=E\int^{\infty}_0e^{-\lambda(s+t)}f(X^y_s)\,ds\big|_{y=X^x_t}.
\end{equation}
We also have 
\begin{equation}
\label{eq4.48}
E\Big(\int^{\infty}_te^{-\lambda s}g(X^x_s)\,d|K^x|_s\,\big|\FF_t\Big)
=E\int^{\infty}_0e^{-\lambda(s+t)}g(X^y_s)\,d|K^y|_s\big|_{y=X^x_t}
\end{equation}
and
\begin{align}
\label{eq4.45}
&E\Big(S_{\infty}(g,X^x,K^x)-S_t(g,X^x,K^x)\,\big|\,\FF_t\Big)\nonumber\\
&\qquad=E\sum_{0< s<\infty}e^{-\lambda(s+t)}
\int^{|\Delta K^y_s|}_0 \bar g(X^y_s-r{\mathbf n}(X^y_s))\,dr\big|_{y=X^x_t}\,.
\end{align}
Indeed, in view of (\ref{eq2.24}), equality (\ref{eq4.48}) will be proved once we prove that
\[
E\Big(\int^{\infty}_te^{-\lambda s}h(X^x_s)\,dK^x_s\,\big|\FF_t\Big)
=E\int^{\infty}_0e^{-\lambda(s+t)}h(X^y_s)\,dK^y_s\big|_{y=X^x_t}
\]
for bounded $h$. But this follows from the Markov property (\ref{eq3.9}) and standard arguments. Let $K^{x,d}_t=\sum_{0<s\le t}\Delta K^x_s$ and for $n\ge1$ let $\{s^n_i\}_{i\ge1}$ be a sequence of partitions of $[0,\infty)$ such that $\max_{i\ge1}(s^n_i-s^n_{i-1})\rightarrow0$ as $n\rightarrow\infty$. Then
\begin{align}
\label{eq4.46}
&S_{\infty}(g,X^x,K^x)-S_t(g,X^x,K^x)\nonumber\\
&\qquad=\lim_{n\rightarrow\infty}\sum^{\infty}_{i=1}e^{-\lambda(t+s^n_i)} \int_0^{|K^{x,d}_{t+s^n_i}-K^{x,d}_{t+s^n_{i-1}}|}
\bar g(X^x_{t+s^n_{i}}-r{\mathbf n}(X^x_{t+s^n _{i}}))\,dr.
\end{align}
Furthermore, $K^{x,d}$ enjoys the Markov property in the sense that (\ref{eq3.9}) holds with $K^x$ replaced by $K^{x,d}$. From this and (\ref{eq4.46}) one can deduce (\ref{eq4.45}). By (\ref{eq4.47}) and (\ref{eq4.43})--(\ref{eq4.45}),
\begin{align}
\label{eq4.20}
u(X^x_0)&=e^{-\lambda t}u(X^x_t)+ \int^t_0e^{-\lambda s}
(f(X^{x}_s)\,ds+g(X^{x}_s)\,d|K^x|_s) \nonumber \\
&\quad+S_t(g,X^x,K^x)-M^{x}_t +S_0(g,X^x,K^x)-M^{x}_t,\quad t\ge0.
\end{align}
Let $\varphi\in C^2(\BR^d)$ and
\[
\bar Y^x_t=e^{-\lambda t}\varphi(X^x_t),\quad t\ge0.
\]
Integrating by parts we get
\[
e^{-\lambda t}\varphi(X^x_t)-\varphi(x)=-\lambda\int^t_0e^{-\lambda s}\varphi(X^{x}_s)\,ds
+\int^t_0e^{-\lambda s}\,d\varphi(X^x_s).
\]
Moreover, by It\^o's formula (see, e.g., \cite[Theorem II.33]{Pr}),
\begin{align*}
\varphi(X^x_t)-\varphi(X^x_0)&=\sum^d_{i=1}\int^t_0\frac{\partial\varphi}{\partial x_i} (X^x_{s-})\,dX^x_s
+\frac12\sum^d_{i,j=1}\int^t_0\frac{\partial^2\varphi}{\partial x_i\partial x_j} (X^x_{s-})\,d[X^{x,i},X^{x,j}]^c_s\\
&\quad+\sum_{0<s\le t}\Big\{\varphi(X^x_s)-\varphi(X^x_{s-})
-\sum^d_{i=1}\frac{\partial\varphi}{\partial x_i}(X^x_{s-})
\Delta X^{x,i}_s\Big\}.
\end{align*}
Hence (see \cite[Theorem 4.4.7]{A}) 
\begin{align}
\label{eq4.21}
\bar Y^x_0&=e^{-\lambda t}\varphi(X^x_t)+\int^t_0e^{-\lambda s}(-L\varphi(X^x_s)+\lambda\varphi(X^x_s))\,ds\nonumber\\
&\quad-\sum^d_{i=1}\int^t_0e^{-\lambda s}\frac{\partial\varphi}{\partial x_i} (X^x_{s-})\,dK^{x,i}_s +M^{\varphi}_t,\quad t\ge0,
\end{align}
for some uniformly integrable martingale $M^{\varphi}$. Suppose now that
$\varphi\in\BR^d$ and $x\in D$ is a global maximum point of $u-\varphi$. We can and will assume that $u(x)=\varphi(x)$. 
Suppose also that $-L\varphi(x)+\lambda u(x)>f(x)$. Since $\varphi\in C^2(\BR^d)$ and the functions $f, u$ are  continuous, there is $r>0$ such that
\begin{equation}
\label{eq4.14}
-L\varphi(y)+\lambda \varphi(y)-f(y)>0,\quad y\in\bar B(x,r).
\end{equation}
Set $\tau=\inf\{t\ge0:|X^x_t-x|>r\}\wedge(\inf\{t\ge0:X^x_t\not\in D\}/2)$. Then from (\ref{eq4.20}) it follows that
\[
u(x)=E\Big(e^{-\lambda\tau}u(X^x_{\tau})+\int^{\tau}_0e^{-\lambda s}f(X^x_s)\,ds
\Big),
\]
and from (\ref{eq4.21}) we get
\begin{align*}
E\bar Y^x_0&=E\Big(e^{-\lambda\tau}\varphi(X^x_{\tau})+\int^{\tau}_0e^{-\lambda s} (-L\varphi(X^x_s)+\lambda\varphi(X^x_s))\,ds\Big).
\end{align*}
Since $u(X^x_{\tau})\le\varphi(X^x_{\tau})$, from the above two equalities and (\ref{eq4.14})
we get $u(x)<E\bar Y^x_0=\varphi(x)$, which is a contradiction. Consider now the case
$x\in\partial D$. As before, assume that $\varphi\in C^2(\BR^d)$,
$x\in \partial D$ is a global maximum point of $u-\varphi$ and $u(x)=\varphi(x)$. Suppose that
\[
\min\big(-L\varphi+\lambda u(x)-f(x),
-\frac{\partial\varphi}{\partial{\mathbf n}}(x)-g(x):=3\varepsilon\big)>0.
\]
Choose  $r>0$ such that (\ref{eq4.14}) is satisfied.
Clearly, for any $y\in\partial D$, $z\in\bar D$ we have
\begin{align*}
\Big|\sum^d_{i=1}\frac{\partial\varphi}{\partial x_i}(z){\mathbf n}_i(y) -\frac{\partial\varphi}{\partial{\mathbf n}}(x)\Big|
&\le\Big|\sum^d_{i=1}\frac{\partial\varphi}{\partial x_i}(z)
({\mathbf n}_i(y)-{\mathbf n}_i(x))\Big|\\
&\quad+\Big|\sum^{d}_{i=1}\Big(\frac{\partial\varphi}{\partial x_i}(z)
-\frac{\partial\varphi}{\partial x_i}(x)\Big){\mathbf n}_i(x)\Big|.
\end{align*}
Since $\varphi\in C^2(\BR^d)$ and the functions $g$ and
$\partial D\ni y\mapsto{\mathbf n}_i(y)$ are continuous, one can find $\delta\in(0,r)$ such that the left-hand side of the above inequality is less then $\varepsilon$ and $|g(y)-g(x)|\le\varepsilon$ for all $z\in\bar B(x,\delta)$, $y\in\bar B(0,\delta)\cap\partial D$. Hence
\begin{equation}
\label{eq4.19}
-\sum^d_{i=1}\frac{\partial\varphi}{\partial x_i}(z){\mathbf n}_i(y) -g(y)\ge\varepsilon,
\quad z\in\bar B(x,\delta),\, y\in\bar B(x,\delta)\cap\partial D.
\end{equation}
Set $\tau=\inf\{t\ge0:|X^{x}_t-x|>\delta\mbox{ or }|\Delta K^x_t|>\delta\}\wedge T$ for some $T>0$. By  (\ref{eq4.20}),
\begin{align}
\label{eq4.29}
u(x)&=E\Big(e^{-\lambda\tau}u(X^x_{\tau})+\int^{\tau}_0e^{-\lambda s}f(X^x_s)\,ds\nonumber\\
&\quad+\int^{\tau}_0e^{-\lambda s}g(X^x_s)\,d|K^x|_s+S_{\tau}(g,X^x,K^x)
-S_0(g,X^x,K^x)\Big),
\end{align}
and by (\ref{eq4.21}),
\begin{align}
\label{eq4.30}
E\bar Y^x_0&=E\Big(e^{-\lambda\tau}\varphi(X^x_{\tau})+\int^{\tau}_0e^{-\lambda s} (-L\varphi(X^x_s)+\lambda\varphi(X^x_s))\,ds\nonumber\\
&\quad-\sum^d_{i=1}\int^{\tau}_0e^{-\lambda s}\frac{\partial\varphi}{\partial x_i}
(X^x_{s-}){\mathbf n}_i(X^x_s)\fch_{\{X^x_s\in\partial D\}}\,d|K^{x}|_s\Big).
\end{align}
Furthermore,
\[
E(S_{\tau}(g,X^x,K^x)-S_{\tau}(g,X^x,K^x))\le\|g\|_{\infty}\,E\sum_{0< s\le\tau}
e^{-\lambda s}|\Delta K_s|\fch_{\{|\Delta K_s|\le\delta\}}.
\]
Since $E|K|_T<\infty$, the right-hand side of the above inequality converges to zero as $\delta\downarrow0$.
Since $u(X^x_{\tau})\le\varphi(X^x_{\tau})$, it follows from this and (\ref{eq4.14})--(\ref{eq4.30})
that for a sufficiently small $\delta>0$ we have
$u(x)<E\bar Y^x_0=\varphi(x)$, which is a contradiction.
Finally, consider the case $x\in\BR^d\setminus\bar D$. Suppose that $x$ is a global maximum point of $u-\varphi$. We may and will assume that $u(x)-\varphi(x)=0$. We then have
$u(y)-\varphi(y)\le u(x)-\varphi(x)=0$ for all $y$ in some neighborhood $x$.
Therefore there is $t_0$ such that for all  $t\in(-t_0,t_0)$,
\[
\varphi(x+t\cdot\bar{\mathbf n}(x))-\varphi(x)
\ge u(x+t\cdot\bar{\mathbf n}(x))-u(x).
\]
Since $\Pi(x+t\cdot\bar{\mathbf n}(x))=\Pi(x)$,  it follows from the above inequality and (\ref{eq4.16}) that
\begin{align*}
\lim_{t\rightarrow0}\frac{\varphi(x+t\cdot\bar{\mathbf n}(x))-\varphi(x)}{t}&\ge
-\lim_{t\rightarrow0}\frac{1}{t}\int_{|x+t\bar{\mathbf n}(x)-\Pi(x)|}^{|x-\Pi(x)|}
g(\Pi(x)-s{\mathbf n}(\Pi(x)))\,ds\\
&=-g(\Pi(x)-|x-\Pi(x)|\cdot{\mathbf n}(\Pi(x)))=-g(x).
\end{align*}
Hence $\frac{\partial \varphi}{\partial\bar{\mathbf n}}(x)\ge-g(x)$, which in view of Remark \ref{rem4.11} completes the proof that $u$ is a viscosity subsolution of (\ref{eq1.1}). The proof that $u$ is a supersolution is similar, so we omit it.
\end{proof}


\begin{remark}
\label{rem4.12}
If $g=0$ and $f$ is continuous and bounded, then
\[
\lim_{n\rightarrow\infty}E\int_0^Te^{-\lambda t}f(X^{x,n}_t)\,dt
=E\int_0^Te^{-\lambda t}f(X^{x}_t)\,dt
\]
under assumptions (A1), (A2) ((A3) is superfluous because we need not estimate the right-hand side of (\ref{eq4.11}), and therefore we need not use Theorem \ref{th3.3}, Proposition \ref{prop3.5}. Similarly, (\ref{eq4.12}) holds true under (A1), (A2). As a result, if (A1), (A2) are satisfied, $f$ is continuous and bounded and $g=0$, then $u_n$ is a continuous viscosity solution and $u_n(x)\rightarrow u(x)$ for $x\in\BR^d$.
\end{remark}

\subsection{Stability results}

In view of Proposition \ref{prop4.10}, formula (\ref{eq4.22}) provides a stochastic representation of the viscosity solution of (\ref{eq1.1}). This representation allows studying  properties of viscosity solutions by stochastic methods. We illustrate this possibility by showing some stability results for (\ref{eq1.1}).

\subsubsection{Perturbations of  local operators}

In what follows $\|\cdot\|$ denote the usual norm in the space of linear operators from $\BR^d$ into $\BR^d$.

\begin{theorem}
\label{th4.41}
Assume that \mbox{\rm (A1)--(A4)} are satisfied. Let $\sigma^n,b^n,f^n,g^n$, $n\ge1$ be continuous functions such that
\[
\lim_{n\rightarrow\infty}
\sup_{x\in\bar D} (\|\sigma^n(x)-\sigma(x)\|+|b^n(x)-b(x)|+|f^n(x)-f(x)|)+\|g^n-g\|_\infty=0.
\]
Let $u$ be the probabilistic solution of \mbox{\rm(\ref{eq1.1})} and $u^n$ be the probabilistic solution of \mbox{\rm(\ref{eq1.1})}  with  $\sigma,b,f,g$ replaced by $\sigma^n,b^n,f^n,g^n$. Then  for every compact subset  $K\subset\BR^d$,
\[
\limsup_{n\rightarrow\infty}\sup_{x\in K}|u^n(x)- u(x)|=0.
\]
\end{theorem}
\begin{proof} The problem is very simple in the case where $\sigma^n=\sigma$  and $b^n=b$. Set $\varepsilon_n=\sup_{x\in\bar D}|f^n(x)-f(x)|+\|g^n-g\|_\infty$, then  by Remark \ref{rem4.2}(i),
\begin{align*}
|u^n(x)-u(x)|&\leq \varepsilon_n\Big(\int_0^\infty e^{-\lambda t}\,dt
+E\int_0^\infty e^{-\lambda t}\,d|K^{\Pi(x)}|_t\Big)\\
&\qquad+ \varepsilon_n|x-\Pi(x)|+2r\varepsilon_nE\sum_{0< t<\infty}e^{-\lambda t}|\Delta K^{\Pi(x)}_t|\\
&\leq \varepsilon_n\Big(\frac{1}{\lambda}+3\lambda\int_0^\infty e^{-\lambda t} E|K^{\Pi(x)}|_t\,dt\Big)+\varepsilon_n|x-\Pi(x)|.
\end{align*}
By Corollary \ref{cor3.6}, the integral on the right-hand side is bounded by some constant for every $x$ ($\Pi(x)$ belongs to the compact set  $\bar D$), which implies that for every  compact set $K$,
\[
\sup_{x\in K}|u^n(x)- u(x)|\leq \varepsilon_n(\mbox{const}+\sup_{x\in K}|x-\Pi(x)|)\rightarrow0.
\]
In the general case we  use the arguments from  the proof of Theorem \ref{th4.4}. We will show that
\[
\lim_{n\rightarrow\infty}u^n(x_n)= u(x)
\]
for every $x$  and every sequence $\{x_n\}$ such that  $x_n\to x$.
Let $(X^{x_n,(n)},K^{x_n,(n)})$ be the unique solution of  the  Skorokhod problem associated with  $Y^{x_n,(n)}$ defined by
\[
Y^{x_n,(n),i}_t=x_{n,i}+\sum^d_{j=1}\int^t_0\sigma^n_{ij}(X^{x_n,(n)}_s)\,dW^j_s
+\int^t_0b^n_i(X^{x_n,(n)}_s)\,ds  +N^i_t,\quad t\ge0,
\]
and let $(X^x,K^x)$ be the unique solution of (\ref{eq3.1}), i.e.  the unique solution of  the  Skorokhod problem associated with  $Y^{x}$ defined by (\ref{eq4.07}). By \cite[Theorem 4, Corollary 11]{Sl}, $Y^{x_n,(n)}\rightarrow Y^x$ in probability in  the $J_1$-topology. Moreover, by  Corollary \ref{cor3.6} and the fact that the coefficients $\sigma^n,b^n$, $n\ge1$, are uniformly bounded,
\begin{equation}
\label{eq4.33}
\sup_{n\ge1}\max\big(E|K^{x_n,(n)}|^p_T,E\sup_{s\leq T}|X^{x_n,(n)}_s|^p\big)\leq C(1+T^p).
\end{equation}
Hence, by Theorem \ref{th2.10}, for every $T>0$,
\[
\lim_{n\rightarrow\infty}E\int_0^Te^{-\lambda t}f^n(X^{x_n,(n)}_t)\,dt
=E\int_0^Te^{-\lambda t}f(X^x_t)\,dt
\]
and
\[
\lim_{n\rightarrow\infty}EI_T(g^n_\lambda,X^{x_n,(n)},K^{x_n,(n)})
=EI_T(g_\lambda,X^{x},K^{x})),
\]
where $g^n_\lambda(t,x)=e^{-\lambda t}g^n(x)$, $t\geq0$, $x\in\BR^d$. To complete the proof it is sufficient to use once again (\ref{eq4.33}) and the arguments from the proof of Theorem \ref{th4.4}.
\end{proof}

\subsubsection{Equations involving  fractional Laplace operators}

Consider the case where $\II=-(-\Delta)^{\alpha/2}$ for some $\alpha\in(1,2)$. It is known that this operator corresponds to the  measure
\[
\nu_{\alpha}(y)=c_{d,\alpha}|y|^{-d-\alpha}\quad\mbox{with}\quad
c_{d,\alpha}=\frac{\Gamma((d-\alpha)/2)}{2^{\alpha}\pi^{d/2}\Gamma(\alpha/2)}.
\]
A direct calculation shows that then (A3) is satisfied if and only if $\alpha\in(1,2)$.
Let $u_{\alpha}$ be the viscosity solution of the problem
\begin{equation}
\label{eq4.35}
(-\Delta)^{\alpha/2}u_{\alpha}+\lambda u_{\alpha}=f, \qquad
\frac{\partial u_{\alpha}}{\partial\bar{\mathbf n}}=-g\quad\mbox{in }D^c,
\end{equation}
and $u$ be the viscosity solution of the problem
\begin{equation}
\label{eq4.36}
-\Delta u+\lambda u=f, \qquad
\frac{\partial u}{\partial{\mathbf n}}=-g\quad\mbox{on }\partial D.
\end{equation}

\begin{proposition}
\label{prop4.14}
Assume that \mbox{\rm(A4)} is satisfied. If $\alpha\rightarrow2^{-}$, then
\[
\lim_{\alpha\rightarrow2^{-}}\sup_{x\in\bar D}|u_{\alpha}(x)- u(x)|=0.
\]
\end{proposition}
\begin{proof}
It is suffices to show that $u_{\alpha_n}(x_n)\rightarrow u(x)$
for every $x\in\bar D$  and all sequences  $\{\alpha_n\}\subset(1,2)$, $\{x_n\}\subset\bar D$ such that $\alpha_n\to 2^{-}$ and  $x_n\to x$.
Let $Y^{\alpha_n}:=x_n+N^{\alpha_n}$, where $N^{\alpha_n}=(N^{\alpha_n,1},\dots, N^{\alpha_n,d})$ denote a symmetric $d$-dimensional $\alpha$-stable process, and let $Y_t=x+\sqrt{2}B_{t}$, $t\ge0$, where $(B^1,\dots,B^d)$ is a standard $d$-dimensional Wiener process. The processes $N^{\alpha_n}$ have independent increments and for all $\theta\in\BR^1$, $q>0$ and $k=1,\dots,d$ we have
\[
\lim_{n\rightarrow\infty}\sup_{t\le q}
\Big|Ee^{i\theta N^{\alpha_n,k}_t}-Ee^{i\theta\sqrt{2} B^k_{t}}\Big|
=\lim_{n\rightarrow\infty}
\sup_{t\le q}\Big| e^{-t|\theta|^{\alpha_n}}-e^{t|\theta|^2}\Big|=0.
\]
From this and known results (see, e.g., \cite[Corollary 2]{JS}) it follows  that $Y^{\alpha_n}\rightarrow Y$ in distribution in the space $\BD(\BR_+;\BR^{d})$ equipped with  the $J_1$-topology. Applying now Theorem \ref{th2.9} we get (\ref{eq4.04}) and (\ref{eq4.05}). Since the estimate (\ref{eq4.06}) is uniform in $\alpha\in(p,2)$ for some fixed $p\in(1,2)$,  we  also have (\ref{eq4.32}) and (\ref{eq4.34}) with $y=x$. From these results we deduce that the probabilistic solution of (\ref{eq4.35}) converges to the probabilistic solution of (\ref{eq4.36}), which proves the proposition in view of Propositions \ref{prop4.3} and \ref{prop4.10}.
\end{proof}

\begin{remark}
It is known that $\lim_{\alpha\rightarrow 2^{-}} \frac{c_{d,\alpha}}{\alpha(2-\alpha)}=\frac{d}{\omega_{d-1}}$, where $\omega_{d-1}$ denotes the $(d-1)$-dimensional measure of the unit sphere $S^{d-1}$  (see \cite[Corollary 4.2]{DPV}). Therefore replacing $c_{d,\alpha}$ by $2-\alpha$ in the definition of $\nu_{\alpha}$ we get as limit of $u_{\alpha}$ the solution $u$ of (\ref{eq4.36}) with the operator $\Delta$ replaced by $\frac{\omega_{d-1}}{d}\Delta$. For results of this type in the half space (and different concepts of reflection)   we refer the reader to \cite{BCGJ}.
\end{remark}


The stability result of  Proposition \ref{prop4.14} can be extended to the case of viscosity solutions of the problems
\begin{equation}
\label{eq4.37}
-\LL u_{\alpha}-(-\Delta)^{\alpha/2}u_{\alpha}+\lambda u_{\alpha}=f, \qquad
\frac{\partial u_{\alpha}}{\partial\bar{\mathbf n}}=-g\quad\mbox{in }D^c.
\end{equation}
Below we shall see that in that case the limit equation has the form (\ref{eq1.9})
with $\tilde \LL$ defined by
\[
\tilde\LL=\frac12\sum^d_{i,j=1}\tilde a_{ij}(x)
\frac{\partial^2}{\partial x_i\partial x_j}
+\sum^d_{i=1}b_i(x)\frac{\partial}{\partial x_i},
\]
where $\tilde a_{ij}=a_{ij}$ for $i\neq j$  and $\tilde a_{ii}=a_{ii}+2$,  $i=1,\dots,d$.

\begin{theorem}
\label{th4.16}
Assume that \mbox{\rm(A4)} is satisfied.
Let $u_{\alpha}$ be the viscosity solution of \mbox{\rm(\ref{eq4.37})} and
$u$ be the viscosity solution of \mbox{\rm(\ref{eq1.9})}. Then
\[
\lim_{\alpha\rightarrow2^{-}}\sup_{x\in\bar D}|u_{\alpha}(x)- u(x)|=0.
\]
\end{theorem}
\begin{proof}
It is sufficient to show that
$u_{\alpha_n}(x_n)\rightarrow u(x)$ for $\{\alpha_n\}, \{x_n\}$ as in the proof
of Proposition \ref{prop4.14}.
Let $(X^{x_n,\alpha_n},K^{x_n,\alpha_n})$ be the unique solution of  the  Skorokhod problem associated with  $Y^{x_n,\alpha_n}$ defined by
\[
Y^{x_n,\alpha_n,i}_t=x_{n,i}+\sum^d_{j=1}\int^t_0\sigma_{ij}(X^{x_n,\alpha_n}_s)\,dW^j_s
+\int^t_0b_i(X^{x_n,\alpha_n}_s)\,ds  +N^{\alpha_n,i}_t,\quad t\ge0.
\]
From the proof of Proposition \ref{prop4.14}  we know that  $N^{\alpha_n}\rightarrow\sqrt{2}B$ in distribution in the space $\BD(\BR_+;\BR^{d})$ equipped with  the $J_1$-topology, where $B$ is a standard $d$-dimensional Brownian motion. Since for every $n\ge1$ the processes $N^{\alpha_n}$  and  $W$  are independent, we have in fact the joint convergence
$(W,N^{\alpha_n})\rightarrow(\bar W,\sqrt{2}B)$ in distribution in the space $\BD(\BR_+;\BR^{2d})$ in the $J_1$-topology, where $\bar W$ is a $d$-dimensional standard Brownian motion independent  of $B$. By \cite[Theorem 4]{Sl}, $(X^{x_n,\alpha_n},K^{x_n,\alpha_n})\rightarrow(\tilde X^{x},\tilde K^{x})$ in distribution in $\BD(\BR_+;\BR^{2d})$ with the  $J_1$-topology, where $(\tilde X^{x},\tilde K^{x})$
is the solution of the  Skorokhod problem associated with  $\tilde Y^x$ defined by
\[
\tilde Y^{x,i}_t=x_{i}+\sum^d_{j=1}\int^t_0\sigma_{ij}(\tilde X^{x}_s)\,d\bar W^j_s
+\int^t_0b_i(\tilde X^{x}_s)\,ds  +\sqrt{2}B^{i}_t,\quad t\ge0.
\]
Clearly, $ Y^{x_n,\alpha_n}=X^{x_n,\alpha_n}-K^{x_n,\alpha_n}$ tends in distribution in the $J_1$-topology to  $\tilde Y^x=\tilde X^x-\tilde K^x$.
Set $\tilde W=(\bar W,B)$, $\tilde \sigma=(\sigma,\sqrt{2}I_d)$, where $I_d$ is the $d$-dimensional identity matrix,  and $\tilde a=\tilde \sigma\cdot\tilde \sigma ^T$. Then $\tilde W$ is a $2d$-dimensional standard Brownian motion,
$\tilde \sigma$  is a $d\times2d$-dimensional array of coefficients and the process $\tilde Y^x$ can be written in the  equivalent form
\[
\tilde Y^{x,i}_t=x_{i}+\sum^{2d}_{j=1}\int^t_0\tilde\sigma_{ij}(\tilde X^{x}_s)\,d\tilde W^j_s
+\int^t_0b_i(\tilde X^{x}_s)\,ds,\quad t\ge0.
\]
To  complete the proof it is sufficient to observe that the  viscosity solution of (\ref{eq1.9}) has the probabilistic representation
\[
u(x)=E\int_0^\infty e^{-\lambda t}(f(\tilde X^{x}_t)\,dt
+g(\tilde X^{x}_t)\,d|\tilde K^{x}|_t),\quad x\in\bar D,
\]
and to repeat arguments from the  proofs of Theorem \ref{th4.41} and Proposition \ref{prop4.14}.
\end{proof}

\end{document}